\documentclass{amsart}
\usepackage{amssymb,mathtools,dsfont,nicefrac}
\usepackage[all]{xy}
\usepackage[pdftitle={Excision in Hochschild and cyclic homology without continuous linear sections}, pdfauthor={Ralf Meyer}, pdfsubject={Mathematics}]{hyperref}
\usepackage[lite]{amsrefs}
\usepackage{enumitem}
\usepackage{microtype}

\newcommand*{\MRref}[2]{ \href{http://www.ams.org/mathscinet-getitem?mr=#1}{MR \textbf{#1}}}
\newcommand*{\arxiv}[1]{\href{http://www.arxiv.org/abs/#1}{arXiv: #1}}

\newtheorem{theorem}{Theorem}[section]
\newtheorem{corollary}[theorem]{Corollary}
\newtheorem{lemma}[theorem]{Lemma}
\newtheorem{proposition}[theorem]{Proposition}

\theoremstyle{remark}
\newtheorem{example}[theorem]{Example}
\newtheorem{remark}[theorem]{Remark}
\theoremstyle{definition}
\newtheorem{definition}[theorem]{Definition}

\DeclareMathOperator{\HH}{HH}
\DeclareMathOperator{\HC}{HC}
\DeclareMathOperator{\HP}{HP}
\DeclareMathOperator{\HHchain}{\mathsf{HH}}
\DeclareMathOperator{\HPchain}{\mathsf{HP}}

\DeclareMathOperator{\Hom}{Hom}
\DeclareMathOperator{\Prec}{Cpt}
\DeclareMathOperator{\diss}{diss}
\DeclareMathOperator{\im}{im}
\DeclareMathOperator{\cone}{cone}

\newcommand*{\nb}{\nobreakdash}
\newcommand*{\blank}{\textup{\textvisiblespace}}
\newcommand*{\defeq}{\mathrel{\vcentcolon=}}
\newcommand*{\mono}{\rightarrowtail}
\newcommand*{\epi}{\twoheadrightarrow}
\newcommand*{\prot}{\mathbin{\hat\otimes_\pi}}
\newcommand*{\hot}{\mathbin{\hat\otimes}}

\newcommand*{\abs}[1]{\lvert#1\rvert}

\newcommand*{\N}{\mathbb N}
\newcommand*{\Z}{\mathbb Z}
\newcommand*{\Q}{\mathbb Q}
\newcommand*{\R}{\mathbb R}
\newcommand*{\C}{\mathbb C}
\newcommand*{\Unit}{\mathds 1}
\newcommand*{\Sphere}{\mathbb S}

\newcommand*{\Cinf}[1][\infty]{\textup C^{#1}}
\newcommand*{\Ccinf}[1][\infty]{\textup C_\textup c^{#1}}
\newcommand*{\Whitney}[1][\infty]{\mathcal E^{#1}}
\newcommand*{\Flat}[3][^\infty]{\mathcal J{#1}(#3;#2)}

\newcommand*{\Cat}{\mathcal C}
\newcommand*{\Abcat}{\mathcal A}
\newcommand*{\Exten}{\mathcal E}
\newcommand*{\Split}{\mathcal E_\oplus}
\newcommand*{\Ab}{\mathfrak{Ab}}
\newcommand*{\Frech}{\mathfrak{Fr}}
\newcommand*{\Born}{\mathfrak{Bor}}
\newcommand*{\Top}{\mathfrak{TVS}}
\newcommand*{\Indban}{\overrightarrow{\mathfrak{Ban}}}
\newcommand*{\Proban}{\overleftarrow{\mathfrak{Ban}}}
\newcommand*{\Indcat}{\overrightarrow{\Cat}}
\newcommand*{\Procat}{\overleftarrow{\Cat}}

\newcommand*{\Id}{\textup{Id}}
\newcommand*{\diff}{\textup d}
\newcommand*{\cpt}{\textup c}
\newcommand*{\dR}{\textup{dR}}
\newcommand*{\Sym}[1]{\textup S_{#1}}
\newcommand*{\leftsub}{\textup l}
\newcommand*{\rightsub}{\textup r}
\newcommand*{\Tang}{\textup T}

\newcommand*{\Ho}{\textup H}
\newcommand*{\Hor}{\mathbf H}

\begin{document}

\title[Excision in Hochschild and cyclic homology]{Excision in Hochschild and cyclic homology\\ without continuous linear sections}
\author{Ralf Meyer}
\email{rameyer@uni-math.gwdg.de}
\address{Mathematisches Institut and Courant Centre ``Higher order structures''\\
  Georg-August Universit\"at G\"ottingen\\
  Bunsenstra{\ss}e 3--5\\
  37073 G\"ottingen\\
  Germany}

\begin{abstract}
  We prove that continuous Hochschild and cyclic homology satisfy excision for extensions of nuclear H\nb-unital Fr\'echet algebras and use this to compute them for the algebra of Whitney functions on an arbitrary closed subset of a smooth manifold.  Using a similar excision result for periodic cyclic homology, we also compute the periodic cyclic homology of algebras of smooth functions and Whitney functions on closed subsets of smooth manifolds.
\end{abstract}
\thanks{Supported by the German Research Foundation (Deutsche Forschungsgemeinschaft (DFG)) through the Institutional Strategy of the University of G\"ottingen}
\maketitle
\tableofcontents

\section{Introduction}
\label{sec:intro}

Hochschild and cyclic homology are invariants of non-commutative algebras that generalise differential forms and de Rham cohomology for smooth manifolds~\cite{Connes:Noncommutative_Diffgeo}.  More precisely, let \(A\defeq \Cinf(X)\) for a smooth manifold~\(X\) be the Fr\'echet algebra of smooth functions on a smooth manifold~\(X\) (we impose no growth condition at infinity).  Its continuous Hochschild homology \(\HH_n(A)\) for \(n\in\N\) is naturally isomorphic to the space of differential \(n\)\nb-forms \(\Omega^n(X)\) on~\(X\).  Its continuous periodic cyclic homology \(\HP_n(A)\) for \(n\in\Z/2\) is naturally isomorphic to the de Rham cohomology of~\(X\) made \(2\)\nb-periodic:
\[
\HP_n\bigl(\Cinf(X)\bigr) \cong \bigoplus_{k\in\Z} \Ho_\dR^{n-2k}(X).
\]
And its continuous cyclic homology \(\HC_n(A)\) interpolates between these two:
\[
\HC_n\bigl(\Cinf(X)\bigr) \cong
\frac{\Omega^n(X)}{\diff\bigl(\Omega^{n-1}(X)\bigr)} \oplus
\bigoplus_{k=1}^\infty {}\Ho_\dR^{n-2k}(X),
\]
where \(\diff\colon \Omega^{n-1}(X) \to \Omega^n(X)\) denotes the de Rham boundary map.  The corresponding continuous cohomology groups \(\HH^n(A)\), \(\HP^n(A)\) and \(\HC^n(A)\) are naturally isomorphic to the topological dual spaces of these Fr\'echet spaces; in particular, \(\HH^n(A)\) is isomorphic to the space of de Rham \(n\)\nb-currents on~\(X\).  Alain Connes~\cite{Connes:Noncommutative_Diffgeo} proves these cohomological results with an explicit projective \(A\)\nb-bimodule resolution of~\(A\).  The same method yields their homological analogues.

Recently, Jean-Paul Brasselet and Markus Pflaum~\cite{Brasselet-Pflaum:Whitney} extended these computations to the algebra of Whitney functions on certain regular subsets of~\(\R^n\).  The proof is quite complicated because the possible singularities of such subsets make it much harder to write down projective bimodule resolutions.  Here we use excision theorems to compute these invariants for the algebra of Whitney functions on \emph{any} closed subset of a smooth manifold.  Our proof is shorter and removes the technical assumptions in~\cite{Brasselet-Pflaum:Whitney}.

Let~\(Y\) be a closed subset of the smooth manifold~\(X\) and let \(\Flat{Y}{X}\) be the closed ideal in \(\Cinf(X)\) consisting of all functions that have vanishing Taylor series at all points of~\(Y\) in some -- hence in any -- local coordinate chart.  The algebra of Whitney functions on~\(Y\) is the Fr\'echet algebra
\[
\Whitney(Y) \defeq \Cinf(X) \bigm/ \Flat{Y}{X}.
\]
It may depend on the embedding of~\(Y\) into~\(X\).  Since \(\Whitney(Y)\) is defined to fit into an extension of Fr\'echet algebras
\begin{equation}
  \label{eq:Whitney_extension}
  \Flat{Y}{X} \mono \Cinf(X) \epi \Whitney(Y),
\end{equation}
we may hope to compute its invariants using the Excision Theorem of Mariusz Wodzicki~\cite{Wodzicki:Excision}, which provides natural long exact sequences in Hochschild, cyclic, and periodic cyclic homology for suitable algebra extensions.

The only issue is whether the Excision Theorem applies to the extension~\eqref{eq:Whitney_extension} because it need not have a continuous linear section, and such a section is assumed by previous excision statements about continuous Hochschild (co)homology of topological algebras; this is why Brasselet and Pflaum use another approach.

Our main task is, therefore, to formulate an Excision Theorem for continuous Hochschild homology that does not require continuous linear sections.  That such a theorem exists has long been known to the experts.  Mariusz Wodzicki stated a special case in \cite{Wodzicki:Cyclic_pdo}*{Proposition 3} and announced general results for topological algebras in \cite{Wodzicki:Excision}*{Remark 8.5.(2)}, which were, however, never published.  The proof of Wodzicki's Excision Theorem by Jorge and Juan Guccione~\cite{Guccione-Guccione:Excision} works in great generality and, in fact, applies to the extension~\eqref{eq:Whitney_extension}, but such generalisations are not formulated explicitly in~\cite{Guccione-Guccione:Excision}.

The example of the algebra of Whitney functions has motivated me to finally state and prove such a general excision theorem here.  I work in a rather abstract categorical setup to avoid further embarrassments with insufficient generality.

The situation in~\cite{Guccione-Guccione:Excision} is a ring extension \(I\mono E\epi Q\) that is \emph{pure}, that is, remains exact after tensoring with another Abelian group, and has an \emph{H\nb-unital} kernel~\(I\).  We generalise the notions of purity and H\nb-unitality to algebras in an additive symmetric monoidal category \((\Cat,\otimes)\) with an exact category structure, that is, a class of distinguished extensions~\(\Exten\), which we call \emph{conflations} following the notation of \cites{Keller:Chain_stable, Keller:Handbook}.  It is routine to check that the arguments in~\cite{Guccione-Guccione:Excision} still work in this generality.

Then we specialise to the category of Fr\'echet spaces with the complete projective topological tensor product and the class of all extensions in the usual sense.  We check that an extension of Fr\'echet spaces with nuclear quotient is pure and that the algebra \(\Flat{Y}{X}\) is H\nb-unital in the relevant sense, so that our Excision Theorem applies to~\eqref{eq:Whitney_extension}.  Excision in Hochschild homology also implies excision in cyclic and periodic cyclic homology.  Thus we compute all three homology theories for the algebra of Whitney functions.

The case of Fr\'echet algebras is our main application.  We also discuss algebras in the categories of inductive or projective systems of Banach spaces, which include complete convex bornological algebras and complete locally convex topological algebras.  For instance, this covers the case of Whitney functions with compact support on a non-compact closed subset of a smooth manifold, which is an inductive limit of nuclear Fr\'echet algebras.

The continuous Hochschild cohomology \(\HH^n(A,A)\) of a Fr\'echet algebra~\(A\) with coefficients in~\(A\) viewed as an \(A\)\nb-bimodule is used in deformation quantisation theory.  For \(A=\Cinf(X)\), this is naturally isomorphic to the space of smooth \(n\)\nb-vector fields on~\(X\), that is, the space of smooth sections of the vector bundle~\(\Lambda^n(\Tang X)\) on~\(X\).  The method of Brasselet and Pflaum~\cite{Brasselet-Pflaum:Whitney} also allows to compute this for the algebra of Whitney functions on sufficiently nice subsets of~\(\R^n\).  I have tried to reprove and generalise this using excision techniques, but did not succeed because purity of an extension is not enough for \emph{cohomological} computations.  While it is likely that the Hochschild cohomology for the algebra of Whitney functions is always the space of Whitney \(n\)\nb-vector fields, excision techniques only yield the corresponding result about \(\HH^n(A,A_k)\), where~\(A_k\) is the \emph{Banach} space of \(k\)\nb-times differentiable Whitney functions, viewed as a module over the algebra~\(A\) of Whitney functions.

\section{Preparations: homological algebra and functional analysis}
\label{sec:homological_algebra_fun}

The abstract framework to define and study algebras and modules are additive \emph{symmetric monoidal categories} (see~\cite{Saavedra:Tannakiennes}).  We discuss some examples of such categories: Abelian groups with their usual tensor product, Fr\'echet spaces with the complete projective tensor product, and inductive or projective systems of Banach spaces with the canonical extensions of the projective Banach space tensor product.

The additional structure of an \emph{exact category} specifies a class of \emph{conflations} to be used instead of short exact sequences.  This allows to do homological algebra in non-Abelian additive categories.  The original axioms by Daniel Quillen~\cite{Quillen:Higher_K} are simplified slightly in~\cite{Keller:Chain_stable}.  We need non-Abelian categories because Fr\'echet spaces or bornological vector spaces do not form Abelian categories.  We describe some natural exact category structures for Abelian groups, Fr\'echet spaces, and inductive or projective systems of Banach spaces.  We also introduce pure conflations~-- conflations that remain conflations when they are tensored with an object.

We show that extensions of nuclear Fr\'echet spaces are always pure and are close to being split in at least two different ways: they remain exact when we apply the functors \(\Hom(V,\blank)\) or \(\Hom(\blank,V)\) for a Banach space~\(V\).  This is related to useful exact category structures on categories of inductive and projective systems.

\subsection{Some examples of symmetric monoidal categories}
\label{sec:examples_smc}

An additive \emph{symmetric monoidal category} is an additive category with a bi-additive tensor product operation~\(\otimes\), a unit object~\(\Unit\), and natural isomorphisms
\begin{equation}
  \label{eq:coherence_tensor}
  (A\otimes B)\otimes C \cong A\otimes (B\otimes C),\qquad A\otimes B\cong B\otimes A,\qquad \Unit\otimes A\cong A\cong A\otimes\Unit
\end{equation}
that satisfy several compatibility conditions (see~\cite{Saavedra:Tannakiennes}), which we do not recall here because they are trivial to check in the examples we are interested in (we do not even specify the natural transformations in the examples below because they are so obvious).  Roughly speaking, the tensor product is associative, symmetric, and monoidal up to coherent natural isomorphisms.  We omit the tensor product, unit object, and the natural isomorphisms above from our notation and use the same notation for a symmetric monoidal category and its underlying category.  The unit object is determined uniquely up to isomorphism, anyway.

The following are examples of additive symmetric monoidal categories:
\begin{itemize}
\item Let~\(\Ab\) be the category of Abelian groups with its usual tensor product~\(\otimes\), \(\Unit=\Z\), and the obvious natural isomorphisms~\eqref{eq:coherence_tensor}.

\item Let~\(\Frech\) be the category of Fr\'echet spaces, that is, metrisable, complete, locally convex topological vector spaces, with continuous linear maps as morphisms.  Let \(\otimes\defeq\prot\) be the complete projective topological tensor product (see~\cite{Grothendieck:Produits}).  Here~\(\Unit\) is~\(\C\) (it would be~\(\R\) if we used real vector spaces).

\item Let~\(\Born\) be the category of complete, convex bornological vector spaces (see~\cite{Hogbe-Nlend:Bornologies}) with bounded linear maps as morphisms.  In the following, all bornological vector spaces are tacitly required to be complete and convex.  Let \(\otimes\defeq\hot\) be the complete projective bornological tensor product (see~\cite{Hogbe-Nlend:Completions}) and let \(\Unit=\C\) once again.

\item Let~\(\Indban\) be the category of inductive systems of Banach spaces.  Let~\(\otimes\) be the canonical extension of the complete projective topological tensor product for Banach spaces to~\(\Indban\): if \((A_i)_{i\in I}\) and \((B_j)_{j\in J}\) are inductive systems of Banach spaces, then \((A_i)_{i\in I} \otimes (B_j)_{j\in J} \defeq (A_i\prot B_j)_{i,j\in I\times J}\).  The unit object is the constant inductive system~\(\C\).

\item Let~\(\Proban\) be the category of projective systems of Banach spaces.  Let~\(\otimes\) be the canonical extension of the complete projective topological tensor product for Banach spaces to~\(\Proban\): if \((A_i)_{i\in I}\) and \((B_j)_{j\in J}\) are projective systems of Banach spaces, then \((A_i)_{i\in I} \otimes (B_j)_{j\in J} \defeq (A_i\prot B_j)_{i,j\in I\times J}\).  The unit object is the constant projective system~\(\C\).

\item Let~\(\Top\) be the category of complete, locally convex, topological vector spaces with the complete projective topological tensor product \(\otimes\defeq\prot\) and \(\Unit=\C\).
\end{itemize}

In each case, the axioms of an additive symmetric monoidal category are routine to check.  Unlike~\(\Frech\), the categories \(\Ab\), \(\Indban\) and~\(\Proban\) are \emph{closed} symmetric monoidal categories, that is, they have an internal \(\Hom\)-functor (see~\cite{Meyer:HLHA}).

The various categories introduced above are related as follows.

First, the precompact bornology functor, which equips a Fr\'echet space with the precompact bornology, is a fully faithful, symmetric monoidal functor
\[
\Prec\colon \Frech\to\Born
\]
from the category of Fr\'echet spaces to the category of bornological vector spaces (see~\cite{Meyer:HLHA}*{Theorems 1.29 and 1.87}).  This means that a linear map between two Fr\'echet spaces is continuous if and only if it maps precompact subsets again to precompact subsets and that the identity map on the algebraic tensor product \(V\otimes W\) of two Fr\'echet spaces \(V\) and~\(W\) extends to an isomorphism \(\Prec(V\prot W) \cong \Prec(V) \hot \Prec(W)\); this amounts to a deep theorem of Alexander Grothendieck~\cite{Grothendieck:Produits} about precompact subsets of \(V\prot W\).

Secondly, there is a fully faithful functor \(\diss\colon \Born\to\Indban\), called \emph{dissection functor}, from the category of bornological vector spaces to the category of inductive systems of Banach spaces.  It writes a (complete, convex) bornological vector space as an inductive limit of an inductive system of Banach spaces in a natural way (see~\cite{Meyer:HLHA}).  It is, unfortunately, not symmetric monoidal on all bornological vector spaces.  The problem is that dissection is not always compatible with completions.  But this pathology rarely occurs.  In particular, it is symmetric monoidal on the subcategory of Fr\'echet spaces by \cite{Meyer:HLHA}*{Theorem 1.166}, that is, the composite functor \(\diss\circ\Prec\colon \Frech\to\Indban\) is a fully faithful and symmetric monoidal functor.  The problems with completions of bornological vector spaces are the reason why local cyclic cohomology requires the category \(\Indban\) instead of \(\Born\) (see~\cite{Meyer:HLHA}).

Explicitly, the functor \(\diss\circ\Prec\colon \Frech\to\Indban\) does the following.  Let~\(V\) be a Fr\'echet space and let~\(I\) be the set of all compact, absolutely convex, circled subsets of~\(V\).  Equivalently, a subset~\(S\) of~\(V\) belongs to~\(I\) if there is a Banach space~\(W\) and an injective, compact linear map \(f\colon W\to V\) that maps the closed unit ball of~\(W\) onto~\(S\).  Given~\(S\), we may take~\(W\) to be the linear span of~\(S\) with the gauge norm of~\(S\).  We denote this Banach subspace of~\(V\) by~\(V_S\).  The set~\(I\) is directed, and \((V_S)_{S\in I}\) is an inductive system of Banach spaces.  The functor \(\diss\circ\Prec\) maps~\(V\) to this inductive system of Banach spaces.

The functor \(\Prec\) extends, of course, to a functor from \(\Top\) to~\(\Born\).  But this functor is neither fully faithful nor symmetric monoidal, and neither is its composition with the dissection functor.

Dually, we may embed \(\Frech\) into~\(\Top\) -- this embedding is fully faithful and symmetric monoidal by definition.  We are going to describe an analogue of the dissection functor that maps \(\Top\) to \(\Proban\) (see~\cite{Prosmans:Derived_analysis}).  Given a locally convex topological vector space~\(V\), let~\(I\) be the set of all continuous semi-norms on~\(I\) and let~\(\hat{V}_p\) for \(p\in I\) be the Banach space completion of~\(V\) with respect to~\(p\).  This defines a projective system \(\diss^*(V)\) of Banach spaces with \(\varprojlim \diss^*(V) = V\) if~\(V\) is complete.  This construction defines a fully faithful functor
\[
\diss^*\colon \Top\to\Proban.
\]
For two complete, locally convex topological vector spaces \(V\) and~\(W\), the semi-norms of the form \(p\prot q\) for continuous semi-norms \(p\) and~\(q\) on \(V\) and~\(W\) generate the projective tensor product topology on \(V\otimes W\).  This yields a natural isomorphism \(\diss^*(V\prot W) \cong \diss^*(V) \otimes \diss^*(W)\), and the functor \(\diss^*\) is symmetric monoidal.

In some situations, it is preferable to use the complete \emph{inductive} topological tensor product on~\(\Top\) (see~\cite{Brodzki-Plymen:Periodic}).  However, this tensor product does not furnish another symmetric monoidal structure on~\(\Top\) because it fails to be associative in general.  It only works on subcategories, such as the category of nuclear LF-spaces, where it is closely related to the projective \emph{bornological} tensor product.

Once we have a symmetric monoidal category, we may define algebras and unital algebras inside this category, and modules over algebras and unitary modules over unital algebras (see~\cite{Saavedra:Tannakiennes}).  Algebras in~\(\Ab\) are rings.  Algebras in \(\Top\) are complete locally convex topological algebras, that is, complete locally convex topological vector spaces~\(A\) with a jointly continuous associative bilinear multiplication \(A\times A\to A\); notice that such algebras need not be locally multiplicatively convex.  Similarly, algebras in \(\Born\) are complete convex bornological algebras, that is, complete convex bornological vector spaces with a (jointly) bounded associative bilinear multiplication.  Unitality has the expected meaning for algebras in \(\Ab\), \(\Top\), and \(\Born\).  If~\(A\) is an algebra in~\(\Ab\), that is, a ring, then left or right \(A\)\nb-modules and \(A\)\nb-bimodules in~\(\Ab\) are left or right \(A\)\nb-modules and \(A\)\nb-bimodules in the usal sense, and unitality for modules over unital algebras has the expected meaning.  The same holds in the categories \(\Top\) and \(\Born\).  A left complete locally convex topological module over a complete locally convex topological algebra~\(A\) is a complete locally convex topological vector space~\(M\) with a jointly continuous left module structure \(A\times M\to M\).

\subsection{Exact category structures}
\label{sec:exact_category}

A pair \((i,p)\) of composable maps \(I\xrightarrow{i}E\xrightarrow{p}Q\) in an additive category is called a \emph{short exact sequence} if~\(i\) is a kernel of~\(p\) and~\(p\) is a cokernel of~\(i\).  We also call the diagram \(I\xrightarrow{i}E\xrightarrow{p}Q\) an \emph{extension} in this case.

\begin{example}
  \label{exa:extension_Frech}
  Extensions in~\(\Ab\) are group extensions or short exact sequences in the usual sense.  The Open Mapping Theorem shows that a diagram of Fr\'echet spaces \(I\to E\to Q\) is an extension in \(\Frech\) if and only if it is exact as a sequence of vector spaces.  This becomes false for more general extensions of bornological or topological vector spaces.
\end{example}

An \emph{exact category} is an additive category~\(\Cat\) with a family~\(\Exten\) of extensions, called \emph{conflations}; we call the maps \(i\) and~\(p\) in a conflation an \emph{inflation} and a \emph{deflation}, respectively, following Keller \cites{Keller:Chain_stable, Keller:Handbook}.  We use the symbols \(\mono\) and~\(\epi\) to denote inflations and deflations, and \(I\mono E\epi Q\) to denote conflations.

The conflations in an exact category must satisfy some axioms (see~\cite{Quillen:Higher_K}), which, thanks to a simplification by Bernhard Keller in the appendix of~\cite{Keller:Chain_stable}, require:
\begin{itemize}
\item the identity map on the zero object is a deflation;

\item products of deflations are again deflations;

\item pull-backs of deflations along arbitrary maps exist and are again deflations; that is, in a pull back diagram
  \[
  \xymatrix{A\ar[r]^{f}\ar[d]&B\ar[d]\\C\ar[r]_g&D,}
  \]
  if~\(g\) is a deflation, so is~\(f\);

\item push-outs of inflations along arbitrary maps exist and are again inflations.
\end{itemize}
These axioms are usually easy to verify in examples.

Any exact category is equivalent to a full subcategory of an Abelian category closed under extensions, such that the conflations correspond to the extensions in the ambient Abelian category.  As a consequence, most results of homological algebra extend easily to exact categories.

We now describe some examples of exact category structures on the symmetric monoidal categories introduced above.

\begin{example}
  \label{exa:split}
  Let~\(\Cat\) be any additive category and let~\(\Split\) be the class of all split extensions; these are isomorphic to direct sum extensions.  This is an exact category structure on~\(\Cat\).
\end{example}

When we do homological algebra with topological or bornological algebras, we implicitly use this trivial exact category structure on the category of topological or bornological vector spaces.  The bar resolutions that we use in this context are all \emph{split} exact (contractible) and their entries are only projective with respect to module extensions that split as extensions of topological or bornological vector spaces.  Of course, our purpose here is to study algebra extensions that are not split, so that we need more interesting classes of conflations.

The following definition is equivalent to one by Jean-Pierre Schneiders~\cite{Schneiders:Quasi-Abelian}.

\begin{definition}
  \label{def:quasi-Abelian}
  An additive category is \emph{quasi-Abelian} if any morphism in it has a kernel and a cokernel and if it becomes an exact category when we take all extensions as conflations.
\end{definition}

In the situation of Definition~\ref{def:quasi-Abelian}, the exact category axioms above simplify slightly (see also \cite{Prosmans:Derived_limits}*{Definition~1.1.3}).  The first two axioms become automatic and can be omitted, and the mere existence of pull-backs and push-outs in the other two axioms is guaranteed by the existence of kernels and cokernels.

It goes without saying that Abelian categories such as~\(\Ab\) are quasi-Abelian.  The category~\(\Top\) is not quasi-Abelian (see~\cite{Prosmans:Derived_analysis}) because quotients of complete topological vector spaces need not be complete.  But the other additive categories introduced above are all quasi-Abelian:

\begin{lemma}
  \label{lem:quasi-Abelian}
  The categories \(\Ab\), \(\Frech\), \(\Born\), \(\Indban\) and \(\Proban\) are quasi-Abelian and hence become exact categories when we let all extensions be conflations.
\end{lemma}

\begin{proof}
  For the categories \(\Ab\), \(\Frech\) and~\(\Born\), we merely describe the inflations and deflations or, equivalently, the strict mono- and epimorphisms and leave it as an exercise to verify the axioms.  The inflations and deflations in \(\Ab\) are simply the injective and surjective group homomorphisms.

  Let \(f\colon V\to W\) be a continuous linear map between two Fr\'echet spaces.  It is an inflation if~\(f\) is a homeomorphism from~\(V\) onto \(f(V)\) with the subspace topology; by the Closed Graph Theorem, this holds if and only if~\(f\) is injective and its range is closed.  The map~\(f\) is a deflation if and only if it is an open surjection; by the Closed Graph Theorem, this holds if and only if~\(f\) is surjective.

  Let \(f\colon V\to W\) be a bounded linear map between two bornological vector spaces.  It is an inflation if and only if~\(f\) is a bornological isomorphism onto~\(f(V)\) with the subspace bornology; equivalently, a subset of~\(V\) is bounded if and only if its \(f\)\nb-image is bounded.  It is a deflation if and only if~\(f\) is a bornological quotient map, that is, any bounded subset of~\(W\) is the \(f\)\nb-image of a bounded subset of~\(V\).

  The category of projective systems over a quasi-Abelian category is again quasi-Abelian by \cite{Prosmans:Derived_limits}*{Proposition 7.1.5}.  Since opposite categories of quasi-Abelian categories are again quasi-Abelian, the same holds for categories of inductive systems by duality.  Since the category of Banach spaces is quasi-Abelian (it is a subcategory of the quasi-Abelian category~\(\Frech\) closed under subobjects, quotients, and conflations), we conclude that \(\Indban\) and \(\Proban\) are quasi-Abelian.

  Furthermore, we can describe the inflations and deflations as follows (see \cite{Prosmans:Derived_limits}*{Corollary 7.1.4}).  A morphism \(f\colon X\to Y\) of inductive systems of Banach spaces is an inflation (or deflation) if and only if there are inductive systems \((X'_i)_{i\in I}\) and \((Y'_i)_{i\in I}\) of Banach spaces and inflations (or deflations) \(f'_i\colon X'_i\to Y'_i\) for all \(i\in I\) that form a morphism of inductive systems, and isomorphisms of inductive systems \(X\cong (X'_i)\) and \(Y\cong (Y'_i)\) that intertwine~\(f\) and~\((f'_i)_{i\in I}\).  A dual statement holds for projective systems.  The same argument shows that a diagram \(I\xrightarrow{i} E\xrightarrow{p} Q\) in \(\Indban\) is an extension if and only if it is the colimit of an inductive system of extensions of Banach spaces.
\end{proof}

\begin{definition}
  \label{def:fully_exact}
  A functor \(F\colon \Cat_1\to\Cat_2\) between two exact categories is \emph{fully exact} if a diagram \(I\to E\to Q\) is a conflation in~\(\Cat_1\) if and only if \(F(I)\to F(E)\to F(Q)\) is a conflation in~\(\Cat_2\).
\end{definition}

The functors \(\Prec\colon \Frech\to\Born\), \(\diss\colon \Born\to\Indban\), and \(\diss^*\colon \Frech\to\Proban\) are fully exact.  Thus \(\diss\circ\Prec\colon \Frech\to\Indban\) and \(\diss^*\colon \Frech\to\Proban\) are fully exact, symmetric monoidal, and fully faithful, that is, they preserve all extra structure on our categories.

The following exact category structures are useful in connection with nuclearity:

\begin{definition}
  \label{def:locally_split}
  An extension \(I\xrightarrow{i} E\xrightarrow{p} Q\) in \(\Indban\) is called \emph{locally split} if the induced sequence
  \[
  \Hom(V,I) \to \Hom(V,E) \to \Hom(V,Q)
  \]
  is an extension (of vector spaces) for each \emph{Banach} space~\(V\) (here we view~\(V\) as a constant inductive sytem).

  An extension \(I\xrightarrow{i} E\xrightarrow{p} Q\) in \(\Proban\) is called \emph{locally split} if the induced sequence
  \[
  \Hom(I,V) \to \Hom(E,V) \to \Hom(Q,V)
  \]
  is an extension (of vector spaces) for each \emph{Banach} space~\(V\).
\end{definition}

It is routine to verify that analogous definitions yield exact category structures for inductive and projective systems over any exact category.  By restriction to the full subcategory~\(\Frech\), we also get new exact category structures on Fr\'echet spaces.

\begin{definition}
  \label{def:locally_split_Frech}
  An extension \(I\xrightarrow{i} E\xrightarrow{p} Q\) of Fr\'echet spaces is \emph{ind-locally split} if any compact linear map \(V\to Q\) for a Banach space~\(V\) lifts to a continuous linear map \(V\to E\) (then the lifting can be chosen to be compact as well).  The extension is called \emph{pro-locally split} if any continuous linear map \(I\to V\) for a Banach space~\(V\) extends to a continuous linear map \(E\to V\).
\end{definition}

It is easy to check that an extension of Fr\'echet spaces is ind-locally split if and only if \(\diss\circ\Prec\) maps it to a locally split extension in \(\Indban\), and pro-locally split if and only if \(\diss^*\) maps it to a locally split extension in \(\Proban\).

\subsection{Exact chain complexes, quasi-isomorphisms, and homology}
\label{sec:exact_qi_Ho}

All the basic tools of homological algebra work in exact categories in the same way as in Abelian categories.  This includes the construction of a derived category (see \cites{Neeman:Derived_Exact, Keller:Handbook}).  To keep this article easier to read, we only use a limited set of tools, however.

\begin{definition}
  \label{def:exact_chain}
  A chain complex \((C_n,d_n\colon C_n\to C_{n-1})\) in an exact category \((\Cat,\Exten)\) is called \emph{exact} (or \(\Exten\)\nb-exact) in degree~\(n\) if \(\ker d_n\) exists and~\(d_{n+1}\) induces a deflation \(C_{n+1} \xrightarrow{d_{n+1}} \ker d_n\).  Equivalently, \(\ker d_n\) and \(\ker d_{n+1}\) exist and the sequence \(\ker d_{n+1} \xrightarrow{\subseteq} C_{n+1} \xrightarrow{d_{n+1}} \ker d_n\) is a conflation.
\end{definition}

In all our examples, the existence of kernels and cokernels comes for free.  In~\(\Ab\), exactness just means \(\im d_{n+1} = \ker d_n\).  By the Open Mapping Theorem, the same happens in~\(\Frech\).  But exactness of chain complexes in \(\Born\), \(\Indban\) and \(\Proban\) requires more than this set theoretic condition.

\begin{example}
  \label{exa:split_quasi-iso}
  Let~\(\Cat\) be an additive category in which all idempotent morphisms have a range object (this follows if all morphisms in~\(\Cat\) have kernels).  Then a chain complex in~\(\Cat\) is \(\Split\)\nb-exact if and only if it is contractible.
\end{example}

\begin{example}
  \label{exa:locally_split_quasi-iso}
  Call a chain complex in \(\Indban\) or \(\Proban\) locally split exact if it is exact with respect to the exact category structures defined in Definition~\ref{def:locally_split}.  A chain complex~\(C_\bullet\) in \(\Indban\) is locally split exact if and only if \(\Hom(V,C_\bullet)\) is exact for each Banach space~\(V\).  Dually, a chain complex~\(C_\bullet\) in \(\Proban\) is locally split exact if and only if \(\Hom(C_\bullet,V)\) is exact for each Banach space~\(V\).
\end{example}

Any symmetric monoidal category~\(\Cat\) carries a \emph{canonical forgetful functor} to the category of Abelian groups,
\[
V\mapsto [V]\defeq \Hom(\Unit,V),
\]
where~\(\Unit\) denotes the tensor unit.  Another \emph{dual} forgetful functor is defined by \([V]^*\defeq \Hom(V,\Unit)\).

\begin{example}
  \label{exa:forgetful_functor}
  The forgetful functor acts identically on~\(\Ab\).  On \(\Frech\) and~\(\Born\), it yields the underlying Abelian group of a Fr\'echet space or a bornological vector space.  The forgetful functors on \(\Indban\) and \(\Proban\) map inductive and projective systems of Banach spaces to their inductive and projective limits, respectively.
\end{example}

The forgetful functor and its dual allow us to define the homology and cohomology of a chain complex in~\(\Cat\).  Let \(\Ho_*(C_\bullet)\) and \(\Ho^*(C_\bullet)\) for a chain complex~\(C_\bullet\) be the homology of the chain complex \([C_\bullet]\) and the cohomology of the cochain complex \([C_\bullet]^*\), respectively.  This yields the usual definition of homology and \emph{continuous} cohomology for chain complexes of Fr\'echet spaces.

While a chain complex of Fr\'echet spaces is exact if and only if its homology vanishes, this fails for chain complexes of bornological vector spaces or for chain complexes in \(\Indban\) and \(\Proban\); there are even exact chain complexes in~\(\Proban\) with non-zero homology.  For Fr\'echet spaces, exactness becomes stronger than vanishing of homology if we use other exact category structures like those in Definition~\ref{def:locally_split_Frech}.  This is remedied by the \emph{refined homology} \(\Hor_*(C_\bullet)\) for chain complexes in~\(\Cat\).

\begin{definition}
  \label{def:refined_homology}
  Recall that any exact category~\(\Cat\) can be realised as a full, fully exact subcategory of an Abelian category (even in a canonical way).  We let~\(\Abcat\) be such an Abelian category containing~\(\Cat\), and we let \(\Hor_n(C_\bullet)\) for a chain complex in~\(\Cat\) be its \(n\)th homology in the ambient Abelian category~\(\Abcat\).
\end{definition}

This refined homology is useful to extend familiar results and definitions from homological algebra to chain complexes in exact categories.

By design \(\Hor_n(C_\bullet)=0\) if and only if~\(C_\bullet\) is exact in degree~\(n\).

We now compare the refined homology with the usual homology for chain complexes of Abelian groups and Fr\'echet spaces.  For chain complexes of Abelian groups, both agree because~\(\Ab\) is already Abelian, so that \(\Hor_n(C_\bullet)\) is the usual homology of a chain complex of Abelian groups.

For a chain complex of Fr\'echet spaces, let \(\Ho_n^\Frech(C_\bullet)\) be its \emph{reduced homology}: the quotient of \(\ker d_n\) by the closure of \(d_{n+1}(C_{n+1})\) with the quotient topology.  This is the Fr\'echet space that comes closest to the homology \(\Hor_n(C_\bullet)\).  Assume that the boundary map~\(d_\bullet\) has closed range.  Then it is automatically open as a map to \(\im d_{n+1}\) with the subspace topology from~\(C_n\); the map \(\ker d_n\to \Ho_n^\Frech(C_\bullet)\) is open, anyway.  Thus \(\ker (d_{n+1}) \mono C_{n+1} \epi \im d_{n+1}\) and \(\im d_{n+1} \mono \ker d_n \epi \Ho_n^\Frech(C_\bullet)\) are conflations of Fr\'echet spaces.  Since the embedding of \(\Frech\) into the ambient Abelian category~\(\Abcat\) is exact, these remain extensions in~\(\Abcat\).  Hence \(\Hor_n(C_\bullet) \cong \Ho_n^\Frech(C_\bullet)\) if~\(d_\bullet\) has closed range.  In general, \(\Hor_n(C_\bullet)\) is the cokernel in the Abelian category~\(\Abcat\) of the map \(C_{n+1} \to \ker d_n\) induced by~\(d_{n+1}\).

\begin{definition}
  \label{def:quasi-isomorphism}
  A \emph{quasi-isomorphism} between two chain complexes in an exact category is a chain map with an exact mapping cone.
\end{definition}

In an Abelian category such as~\(\Ab\), quasi-isomorphisms are chain maps that induce an isomorphism on homology.  As a consequence, a chain map is a quasi-isomorphism if and only if it induces an isomorphism on the \emph{refined} homology.

\begin{lemma}
  \label{lem:quasi-iso_Frechet}
  A chain map between two chain complexes of Fr\'echet spaces is a quasi-isomorphism with respect to the class of all extensions if and only if it induces an isomorphism on homology.
\end{lemma}

\begin{proof}
  The mapping cone of a chain map~\(f\) is again a chain complex of Fr\'echet spaces and hence is exact if and only if its homology vanishes.  By the Puppe long exact sequence, the latter homology vanishes if and only if~\(f\) induces an isomorphism on homology.
\end{proof}

Quasi-isomorphisms in \(\Born\), \(\Indban\), or \(\Proban\) are more complicated to describe.

To prove the excision theorem, we must show that certain chain maps are quasi-isomorphisms.  The arguments in~\cite{Guccione-Guccione:Excision} use homology to detect quasi-isomorphisms and, with our refined notion of homology, carry over literally to any exact symmetric monoidal category.  But, in fact, we do not need this sophisticated notion of homology because we only need quasi-isomorphisms of the following simple type:

\begin{lemma}
  \label{lem:quasi-iso_criterion}
  Let \(I_\bullet\overset{i}\mono E_\bullet \overset{p}\epi Q_\bullet\) be a conflation of chain complexes in~\(\Cat\), that is, the maps~\(i\) and~\(p\) are chain maps and restrict to conflations \(I_n\mono E_n\epi Q_n\) for all \(n\in\Z\).  If~\(I_\bullet\) is exact, then~\(p\) is a quasi-isomorphism.  If~\(Q_\bullet\) is exact, then~\(i\) is a quasi-isomorphism.
\end{lemma}

\begin{proof}
  Our conflation of chain complexes yields a long exact homology sequence for refined homology because this works in Abelian categories.  By exactness, the map~\(i\) induces an isomorphism on refined homology if and only if the refined homology of~\(Q_\bullet\) vanishes.  That is, the map~\(i\) is a quasi-isomorphism if and only if~\(Q_\bullet\) is exact.  A similar argument shows that~\(p\) is a quasi-isomorphism if and only if~\(I_\bullet\) is exact.
\end{proof}

Besides Lemma~\ref{lem:quasi-iso_criterion}, we also need to know that a composite of two quasi-iso\-morphisms is again a quasi-isomorphism -- this follows because refined homology is a functor.  This together with Lemma~\ref{lem:quasi-iso_criterion} suffices to verify the quasi-isomorphisms we need.

Finally, we need a sufficient condition for long exact homology sequences.  Let
\begin{equation}
  \label{eq:cofibre_sequence}
  A \xrightarrow{f} B \xrightarrow{g} C  
\end{equation}
be chain maps between chain complexes in~\(\Cat\) with \(g\circ f=0\).  Then we get an induced chain map from the mapping cone of~\(f\) to~\(C\).  We call~\eqref{eq:cofibre_sequence} a \emph{cofibre sequence} if this map \(\cone(f)\to C\) is a quasi-isomorphism.  The \emph{Puppe exact sequence} provides a long exact sequence relating the refined homologies of \(A\), \(B\) and \(\cone(f)\).  For a cofibre sequence, we may identify the refined homology of~\(C\) with that of~\(\cone(f)\) and thus get a natural long exact sequence
\[
\dotsb \to \Hor_n(A) \xrightarrow{f_*} \Hor_n(B) \xrightarrow{g_*} \Hor_n(C) \to \Hor_{n-1}(A) \xrightarrow{f_*} \Hor_{n-1}(B) \to \dotsb.
\]

We get a corresponding long exact sequence for the unrefined homology provided~\(\Unit\) is a projective object, that is, the canonical forgetful functor is exact.  This is the case for \(\Ab\), \(\Frech\), \(\Born\), and \(\Indban\), but \emph{not} for \(\Proban\) because projective limits are not exact.  Similarly, we get a long cohomology exact sequence if~\(\Unit\) is injective.  This is the case for \(\Ab\), \(\Frech\), and \(\Proban\), but \emph{not} for \(\Born\) and \(\Indban\).  There is no long exact cohomology sequence for arbitrary cofibre sequences in \(\Born\) because the Hahn--Banach Theorem fails for bornological vector spaces.  The dual forgetful functor on \(\Indban\) is not exact because it involves projective limits.

\subsection{Hochschild homology and cohomology}
\label{sec:HH}

Let~\(\Cat\) be a symmetric monoidal category.  Let~\(A\) be an algebra in~\(\Cat\), possibly without unit.  We first define the Hochschild homology and cohomology of~\(A\) with coefficients in an \(A\)\nb-bimodule~\(M\).  Then we define the Hochschild homology and cohomology of~\(A\) without coefficients.

The \emph{Hochschild homology} \(\HH_*(A,M)\) of~\(A\) with coefficients~\(M\) is the homology of the chain complex
\[
\HHchain_*(A,M) =
(M\otimes A^{\otimes n},b) =
\bigl(
\dotsb \to M\otimes A^{\otimes 2}
\xrightarrow{b} M\otimes A
\xrightarrow{b} M \to
0 \to \dotsb
\bigr)
\]
in~\(\Cat\), where~\(b\) is defined by categorifying the usual formula
\begin{multline*}
  b(x_0\otimes x_1\otimes \dotsb\otimes x_n) \\\defeq
  \sum_{j=0}^{n-1} (-1)^j x_0\otimes \dotsb \otimes x_jx_{j+1}
  \otimes \dotsb \otimes x_n
  + (-1)^n x_nx_0\otimes x_1\otimes \dotsb \otimes x_{n-1}
\end{multline*}
for \(x_0\in M\), \(x_1,\dotsc,x_n\in A\).  The formula in the \(j\)th summand corresponds to the map \(\Id_{M\otimes A^{\otimes j-2}}\otimes m\otimes \Id_{A^{\otimes n-j}}\), where \(m\colon A\otimes A\to A\) is the multiplication map; the zeroth summand is the multiplication map \(m_{MA}\colon M\otimes A\to M\) tensored with \(\Id_{A^{\otimes n-1}}\), and the last summand involves the multiplication map \(m_{AM}\colon A\otimes M\to M\) and the cyclic rotation of tensor factors \(M\otimes A^{\otimes n} \to A\otimes M \otimes A^{\otimes n-1}\), which exists in symmetric monoidal categories.
 
The \emph{Hochschild cohomology} \(\HH^*(A,M)\) of~\(A\) with coefficients in~\(M\) is the cohomology of the cochain complex
\begin{multline*}
  \HHchain^*(A,M) \defeq
  (\Hom(A^{\otimes n},M),b^*) \\=
  \bigl(
  \dotsb \to 0 \to \Hom(A^{\otimes 0}, M)
  \xrightarrow{b^*} \Hom(A^{\otimes 1}, M)
  \xrightarrow{b^*} \Hom(A^{\otimes 2}, M)
  \to \dotsb
  \bigr),
\end{multline*}
where we interpret \(A^{\otimes 0}\defeq\Unit\) and define~\(b^*\) by categorifying the usual formula
\begin{multline*}
  b^*f(a_1,\dotsc,a_n) \defeq a_1f(a_2,\dotsc,a_n) +
  \sum_{j=1}^{n-1} (-1)^j f(a_1,\dotsc,a_ja_{j+1},\dotsc,a_n)
  \\+ (-1)^n f(a_1,\dotsc,a_{n-1})a_n.
\end{multline*}

Hochschild homology and cohomology are just Abelian groups.  We may also consider the refined homology of \(\HHchain_*(A,M)\).  Our excision results initially deal with this refined Hochschild homology.  This carries over to the unrefined theories if~\(\Unit\) is projective.  If the symmetric monoidal category~\(\Cat\) is closed, we may replace~\(\Hom\) by the internal \(\Hom\)-functor to enrich \(\HHchain^*(A,M)\) to a cochain complex in~\(\Cat\).  This provides a refined version of Hochschild cohomology.

If~\(M\) is a right \(A\)\nb-module, then we may turn it into an \(A\)\nb-bimodule by declaring the left multiplication map \(A\otimes M\to M\) to be the zero map.  This has the effect that the last summand in the map~\(b\) vanishes, so that \(b\) reduces to the map usually denoted by~\(b'\).  Hence assertions about \((M\otimes A^{\otimes n},b)\) for \emph{bi}modules~\(M\) contain assertions about \((M\otimes A^{\otimes n},b')\) for \emph{right} modules~\(M\) as special cases.  Similarly, we may enrich a left \(A\)\nb-module~\(M\) to a bimodule using the zero map \(M\otimes A\to M\), and our assertions specialise to assertions about \((A^{\otimes n}\otimes M,b')\) for left \(A\)\nb-modules~\(M\).

Now we define the Hochschild homology and cohomology of~\(A\) without coefficients.  Let~\(\Unit\) be the tensor unit of~\(\Cat\).  If the algebra~\(A\) is unital, we simply let \(\HH_*(A)\) and \(\HH^*(A)\) be the homology and cohomology of the chain complex \(\HHchain_*(A,A)\).  Thus \(\HH_*(A) \cong \HH_*(A,A)\).  If~\(\Cat\) is closed, then we may form a dual object \(A^*\defeq \Hom(A,\Unit)\) inside~\(\Cat\), and \(\HH^*(A) \cong \HH^*(A,A^*)\).

For a non-unital algebra, the definition involves the unital algebra generated by~\(A\), which is \(A^+\defeq A\oplus\Unit\) with the multiplication where the coordinate embedding \(\Unit \to A^+\) is a unit.  We let \(\HHchain_*(A)\) be the kernel of the augmentation map \(\HHchain_*(A,A^+)\to\Unit\) induced by the coordinate projection \(A^+\to\Unit\).  That is, \(\HHchain_0(A) = A\) and \(\HHchain_n(A) = A^+ \otimes A^{\otimes n}\) for \(n\ge 1\), with the boundary map~\(b\); this is the chain complex of non-commutative differential forms over~\(A\) with the usual Hochschild boundary on non-commutative differential forms.  We let \(\HH_*(A)\) and \(\HH^*(A)\) be the homology and cohomology of \(\HHchain_*(A)\).  It is well-known that \(\HHchain_*(A)\) and \(\HHchain_*(A,A)\) are quasi-isomorphic for unital~\(A\) --~this is a special case of Corollary~\ref{cor:HH_II_EI} below.

Besides \(\HH_*(A)\) and \(\HH^*(A)\), we are also interested in the Hochschild cohomology \(\HH^*(A,A)\), which plays an important role in deformation quantisation and which, in low dimensions, specialises to the centre and the space of outer derivations.

Cyclic homology and periodic cyclic homology can be defined for algebras in~\(\Cat\) as well by carrying over the usual recipes (see also~\cite{Cortinas-Valqui:Excision}).  Since it is well-known anyway that excision in Hochschild (co)homology implies excision in cyclic and periodic cyclic (co)homology, we do not repeat these definitions here.

\subsection{Pure conflations}
\label{sec:pure}

\begin{definition}
  \label{def:pure_extension}
  A conflation \(I\mono E\epi Q\) in~\(\Cat\) is called \emph{pure} if
  \[
  I\otimes V\to E\otimes V\to Q\otimes V
  \]
  is a conflation for all objects~\(V\) of~\(\Cat\).  A chain complex~\(C_\bullet\) is called \emph{pure exact} if \(C_\bullet\otimes V\) is exact for all objects~\(V\) of~\(\Cat\).  A chain map~\(f\) is called a \emph{pure quasi-isomorphism} if \(f\otimes \Id_V\) is a quasi-isomorphism for all objects~\(V\) of~\(\Cat\).
\end{definition}

\begin{definition}
  \label{def:right_exact}
  A functor \(F\colon \Cat_1\to\Cat_2\) between exact categories is called \emph{right exact} if it preserves deflations.  Equivalently, if \(I\overset{i}\mono E\overset{p}\epi Q\) is a conflation in~\(\Cat_1\), then \(\ker\bigl(F(p)\bigr) \to F(E) \xrightarrow{F(p)} F(Q)\) is a conflation in~\(\Cat_2\).
\end{definition}

For instance, a functor \(\Frech\to\Frech\) is right exact if and only if it maps open surjections again to open surjections.

The tensor product functors in the quasi-Abelian categories \(\Ab\), \(\Frech\), \(\Born\), \(\Indban\), and \(\Proban\) are right exact in this sense in each variable.  If the tensor product is right exact, then an extension \(I_\bullet\overset{i}\mono E_\bullet \overset{p}\epi Q_\bullet\) is pure if and only if the natural map \(I\otimes V\to \ker (p\otimes V)\) is an isomorphism for all~\(V\).  In the category~\(\Ab\), this map is always surjective, so that only its injectivity is an issue.  In the categories \(\Frech\) and~\(\Born\), this map is usually injective -- counterexamples are related to counterexamples to Grothendieck's Approximation Property -- and its range is always dense, but it is usually not surjective.

It is clear that split extensions are pure.

We are going to show that any extension of Fr\'echet spaces with nuclear quotient is pure.  This is already known, but we take this opportunity to give two new proofs that use locally split extensions in \(\Indban\) and \(\Proban\), respectively.

\begin{definition}
  \label{def:nuclear}
  An inductive system of Banach spaces \((V_i,\varphi_{ij}\colon V_i\to V_j)_{i\in I}\) is called \emph{nuclear} if for each \(i\in I\) there is \(j\in I_{\ge i}\) for which the map \(\varphi_{ij}\colon V_i\to V_j\) is nuclear, that is, belongs to the projective topological tensor product \(V_i^*\prot V_j\).

  A projective system of Banach spaces \((V_i,\varphi_{ji}\colon V_j\to V_i)_{i\in I}\) is called \emph{nuclear} if for each \(i\in I\) there is \(j\in I_{\ge i}\) for which the map \(\varphi_{ji}\colon V_j\to V_i\) is nuclear.

  A map \(X\to Y\) between two inductive or projective systems of Banach spaces is called \emph{nuclear} if it factors as \(X\to V\to W\to Y\) for a nuclear map between Banach spaces \(V\to W\).
\end{definition}

By definition, an inductive system~\(X\) of Banach spaces is nuclear if and only if each map from a Banach space to~\(X\) is nuclear, and a projective system~\(X\) of Banach spaces is nuclear if and only if each map from~\(X\) to a Banach space is nuclear.

Almost by definition, a bornological vector space~\(V\) is nuclear if and only if \(\diss(V)\) is nuclear in \(\Indban\), and a locally convex topological vector space~\(V\) is nuclear if and only if \(\diss^*(V)\) is nuclear in \(\Proban\) (see~\cite{Hogbe-Nlend-Moscatelli:Nuclear}).  Furthermore, a Fr\'echet space~\(V\) is nuclear if and only if \(\Prec(V)\) is nuclear, if and only if \(\diss \Prec(V)\) is nuclear, see \cite{Hogbe-Nlend-Moscatelli:Nuclear}*{Theorem~(7) on page~160}.

\begin{proposition}
  \label{pro:nuclear_quotient_locally_split}
  Extensions in \(\Indban\) or \(\Proban\) with nuclear quotient are locally split.
\end{proposition}

\begin{proof}
  Let \(I\mono E\epi Q\) be an extension in \(\Indban\) with nuclear~\(Q\).  Recall that we may write it as an inductive limit of extensions of Banach spaces \(I_\alpha\mono E_\alpha \epi Q_\alpha\).  Nuclearity of~\(Q\) means that for each~\(\alpha\) there is \(\beta\ge\alpha\) for which the map \(Q_\alpha\to Q_\beta\) is nuclear.  Now we recall that nuclear maps between Banach spaces may be lifted in extensions.  That is, the map \(Q_\alpha\to Q_\beta\) lifts to a bounded linear map \(s_\alpha\colon Q_\alpha \to E_\beta\) for some \(\beta\ge\alpha\).  Now let~\(V\) be any Banach space.  The space of morphisms from~\(V\) to~\(Q\) is \(\varinjlim \Hom(V,Q_\alpha)\), that is, any morphism \(V\to Q\) factors through a map \(f\colon V\to Q_\alpha\) for some~\(\alpha\).  Then \(s_\alpha\circ f\colon V\to E_\beta\) lifts~\(f\) to a morphism \(V\to E\).  Thus our extension in~\(\Indban\) is locally split if~\(Q\) is nuclear.

  Similarly, an extension in \(\Proban\) is the limit of a projective system of extensions of Banach spaces \(I_\alpha\mono E_\alpha \epi Q_\alpha\).  Nuclearity of~\(Q\) means that for each~\(\alpha\) there is \(\beta\ge\alpha\) for which the map \(Q_\beta\to Q_\alpha\) is nuclear.  As above, this allows us to lift it to a map \(Q_\beta\to E_\alpha\).  Subtracting this map from the canonical map \(E_\beta\to E_\alpha\) yields a map \(E_\beta\to I_\alpha\) that extends the canonical map \(I_\beta\to I_\alpha\).  As above, this shows that any map \(I\to V\) for a Banach space~\(V\) extends to a map \(E\to V\), using that it factors through~\(I_\alpha\) for some~\(\alpha\).
\end{proof}

\begin{proposition}
  \label{pro:locally_split_pure}
  Locally split extensions in \(\Indban\) or \(\Proban\) are pure locally split: if \(I\mono E\epi Q\) is locally split, then \(I\otimes V\mono E\otimes V\epi Q\otimes V\) is a locally split extension as well and \emph{a fortiori} an extension.
\end{proposition}

\begin{proof}
  First we claim that a locally split extension in \(\Indban\) or \(\Proban\) may be written as an inductive or projective limit of extensions that are split, but usually with incompatible sections, so that the limit extension does not split.  We only write this down for inductive systems, the case of projective systems is dual.  An analogous argument works for locally split extensions of projective or inductive systems over any additive category.

  Write a locally split extension as an inductive limit of extensions of Banach spaces \(I_\alpha\mono E_\alpha \epi Q_\alpha\).  For each~\(\alpha\), the canonical map \(Q_\alpha\to Q\) lifts to a map \(Q_\alpha\to E\), which is represented by a map \(s_\alpha\colon Q_\alpha \to E_\beta\) for some \(\beta\ge\alpha\).  For each such pair of indices \((\alpha,\beta)\), we may pull back the extension \(I_\beta\mono E_\beta\epi Q_\beta\) along the map \(Q_\alpha\to Q_\beta\) to an extension \(I_\beta\mono E'_{\beta,\alpha} \epi Q_\alpha\).  The lifting~\(s_\alpha\) induces a section \(Q_\alpha \to E'_{\beta,\alpha}\) for this pulled back extension.  The pairs \((\beta,\alpha)\) above form a directed set and the split extensions \(I_\beta \mono E'_{\beta,\alpha} \epi Q_\alpha\) form an inductive system of extensions indexed by this set; its inductive limit is the given extension \(I\mono E\epi Q\).

  Now we prove the purity assertion.  Write the extension \(I\mono E\epi Q\) as an inductive system of split extensions of Banach spaces \(I_\alpha \mono E_\alpha \epi Q_\alpha\).  Let~\(V\) be another object of~\(\Indban\).  The tensor product in \(\Indban\) commutes with inductive limits, so that \(I\otimes V\cong \varinjlim I_\alpha\otimes V\), and so on.  Since the extensions \(I_\alpha \mono E_\alpha \epi Q_\alpha\) split, so do the extensions \(I_\alpha \otimes V\mono E_\alpha \otimes V\epi Q_\alpha \otimes V\).  This implies that \(I \otimes V \to E\otimes V\to Q\otimes V\) is a locally split extension in \(\Indban\).
\end{proof}

\begin{theorem}
  \label{the:nuclear_Frechet_locally_split_extension}
  Let \(I\mono K\epi Q\) be an extension of Fr\'echet spaces.  If~\(Q\) is nuclear, then \(I\mono E \epi Q\) is both pure ind-locally split and pure pro-locally split.
\end{theorem}

\begin{proof}
  We only write down why \(I\otimes V\to E\otimes V \to Q\otimes V\) is pro-locally split for any Fr\'echet space~\(V\).  A similar argument yields that it is ind-locally split.  If~\(Q\) is a nuclear Fr\'echet space, then \(\diss^*(Q)\) is nuclear in \(\Proban\).  Since the functor \(\diss^*\) is fully exact, it maps \(I\mono E\epi Q\) to an extension in \(\Proban\).  Proposition~\ref{pro:nuclear_quotient_locally_split} asserts that this extension is locally split.  Since~\(\diss^*\) is symmetric monoidal, it maps the diagram \(I\otimes V\to E\otimes V\to Q\otimes V\) to
  \[
  \diss^*(I)\otimes \diss^*(V) \to
  \diss^*(I)\otimes \diss^*(V) \to
  \diss^*(I)\otimes \diss^*(V).
  \]
  This is a locally split extension by Proposition~\ref{pro:locally_split_pure}.  Hence the original diagram was a pro-locally split extension.
\end{proof}

\section{Excision in Hochschild homology}
\label{sec:excision_HH}

We fix a symmetric monoidal category~\(\Cat\) with a tensor product~\(\otimes\) and an exact category structure~\(\Exten\).

\begin{definition}
  \label{def:H-unital_algebra}
  Let~\(A\) be an algebra in~\(\Cat\).  We call~\(A\) \emph{homologically unital}, briefly H\nb-unital, if the chain complex \((A^{\otimes n},b')_{n\ge 1}\) is pure exact.

  Let~\(A\) be an algebra in~\(\Cat\) and let~\(M\) be a right \(A\)\nb-module.  We call~\(M\) \emph{homologically unitary}, briefly H\nb-unitary, if the chain complex \((M\otimes A^{\otimes n},b')_{n\ge 0}\) is pure exact.  A similar definition applies to left \(A\)\nb-modules.
\end{definition}

By definition, the algebra~\(A\) is homologically unital if and only if it is homologically unitary when viewed as a left or right module over itself.

Recall that a chain complex is exact if and only if its refined homology vanishes.  Therefore, \(M\) is homologically unitary if and only if \(\Hor_n\bigl(\HHchain_*(A,M\otimes V)\bigr)=0\) for all~\(V\).  If~\(\Cat\) is an Abelian category or the category of Fr\'echet spaces with all extensions as conflations, then a chain complex is exact if and only if its homology vanishes.  In this case, \(M\) is homologically unitary if and only if \(\HH_*(A,M\otimes V)=0\) for all~\(V\).  In general, H\nb-unitarity is unrelated to the vanishing of \(\HH_*(A,M\otimes V)\).

\begin{remark}
  \label{rem:self-induced}
  Let~\(M\) be a left \(A\)\nb-module.  If the chain complex \((A^{\otimes n} \otimes M,b')_{n\ge0}\) is exact in dimensions zero and one, then the natural map \(A\otimes_A M\to M\) induced by the module structure \(A\otimes M\to A\) is an isomorphism.  If the map \(A\otimes_A A \to A\) is invertible, then~\(A\) is called \emph{self-induced} in~\cite{Meyer:Smooth_rough}; if \(A\otimes_A M\to M\) is an isomorphism, then the \(A\)\nb-module~\(M\) is called \emph{smooth}.  As a result, H\nb-unital algebras are self-induced and H\nb-unitary modules over self-induced algebras are smooth in the sense of~\cite{Meyer:Smooth_rough}.
\end{remark}

\begin{lemma}
  \label{lem:H-unitary_module_extends}
  Let \(I\mono E\epi Q\) be an algebra conflation and let~\(M\) be a homologically unitary \(I\)\nb-module.  Then the \(I\)\nb-module structure on~\(M\) extends uniquely to an \(E\)\nb-module structure on~\(M\).
\end{lemma}

\begin{proof}
  We write down the proof for right modules; similar arguments work for left modules and bimodules.  Since~\(M\) is H\nb-unitary, we get exact chain complexes \((M\otimes I^{\otimes n},b')\) and \((M\otimes I^{\otimes n},b')\otimes E\).  The maps \(M\otimes I^{\otimes n}\otimes E \to M\otimes I^{\otimes n}\) induced by the multiplication map \(I\otimes E\to I\) provide a chain map between these chain complexes above degree~\(0\), that is, we get a commuting diagram
  \[
  \xymatrix{
    \dotsb \ar[r]&
    M\otimes I\otimes I \otimes E \ar[r] \ar[d]&
    M\otimes I \otimes E \ar[r] \ar[d]&
    M \otimes E \ar[r]&
    0\\
    \dotsb \ar[r]&
    M\otimes I\otimes I \ar[r]&
    M\otimes I \ar[r]&
    M \ar[r]&
    0.
  }
  \]
  A right \(E\)\nb-module structure \(M\otimes E\to M\) on~\(M\) extending the given \(I\)\nb-module structure would complete this commuting diagram to a chain map.  Since the rows are exact, there is a unique such completion.  This defines an \(E\)\nb-module structure on~\(M\): associativity follows from the uniqueness of completing another diagram involving maps \(M\otimes I^{\otimes n}\otimes E^{\otimes 2} \to M\otimes I^{\otimes n}\).
\end{proof}

\begin{theorem}
  \label{the:excision_ideal_extension}
  Let \(I\mono E\epi Q\) be a pure algebra conflation, let~\(M\) be an \(E,I\)\nb-bimodule.  Assume that~\(M\) is homologically unitary as a right \(I\)\nb-module and view~\(M\) as an \(E\)\nb-bimodule.  Then the canonical map \(\HHchain_*(I,M) \to \HHchain_*(E,M)\) is a pure quasi-isomorphism.  Thus \(\HH_*(I,M\otimes V) \cong \HH_*(E,M\otimes V)\) for any object~\(V\) of~\(\Cat\) provided~\(\Unit\) is projective.
\end{theorem}

\begin{proof}
  This theorem is an analogue of
  \cite{Guccione-Guccione:Excision}*{Theorem 2} and is proved
  by exactly the same argument.  For \(p\in\N\), let~\(F_p\) be
  the chain complex
  \begin{multline*}
    \dotsb \leftarrow 0 \leftarrow M \otimes V \xleftarrow{b}
    M\otimes E\otimes V
    \xleftarrow{b} M\otimes E\otimes E\otimes V
    \leftarrow \dotsb \leftarrow M\otimes E^{\otimes p}\otimes V
    \\\xleftarrow{b} M \otimes I\otimes E^{\otimes p}\otimes V
    \xleftarrow{b} M \otimes I\otimes I \otimes E^{\otimes p}\otimes V
    \xleftarrow{b} M \otimes I^{\otimes 3} \otimes E^{\otimes p}\otimes V
    \leftarrow \dotsb
  \end{multline*}
  with \(M\otimes V\) in degree~\(0\).

  Since the conflation \(I \mono E\epi Q\) is pure, the map
  \(M\otimes I^{\otimes k}\otimes E^{\otimes p}\otimes V \to
  M\otimes I^{\otimes k-1}\otimes E^{\otimes p+1}\otimes V\) is
  an inflation for all \(k,p\ge0\).  Hence the canonical map
  \(F_p\to F_{p+1}\) is an inflation for each~\(p\).  Its
  cokernel is the chain complex
  \[
  F_{p+1}/F_p \cong (M\otimes I^{\otimes k},b')_{k\ge0}[p+1]
  \otimes Q\otimes E^{\otimes p}\otimes V,
  \]
  where~\([p+1]\) denotes translation by~\(p+1\).  This chain complex is exact because~\(M\) is homologically unitary as a right \(I\)\nb-module.  Since \(F_p\mono F_{p+1}\epi F_{p+1}/F_p\) is a pure conflation of chain complexes, Lemma~\ref{lem:quasi-iso_criterion} shows that the map \(F_p\to F_{p+1}\) is a pure quasi-isomorphism.  Hence so are the embeddings \(F_0\to F_p\) for all \(p\in\N\).  For \(p=0\), we get \(F_0=\HHchain_*(I,M)\otimes V\).  In any fixed degree~\(n\), we have \((F_p)_n=\HHchain_n(E,M)\otimes V\) for \(p\ge n\).  Hence the canonical map \(\HHchain_*(I,M)\to\HHchain_*(E,M)\) is a pure quasi-isomorphism.
\end{proof}

\begin{corollary}
  \label{cor:HH_II_EI}
  Let \(I\mono E\epi Q\) be a pure algebra conflation.  If~\(I\) is homologically unital, then the canonical maps
  \[
  \xymatrix{
    \HHchain_*(I,I) \ar[r]\ar@{=}[d]&
    \HHchain_*(E,I)\ar@{=}[d]\\
    (I^{\otimes n+1},b)\ar[r]&
    (I\otimes E^{\otimes n},b)
  }\qquad
  \xymatrix{
    \HHchain_*(I,I) \ar[r]\ar@{=}[d]&
    \HHchain_*(I)\ar@{=}[d]\\
    (I^{\otimes n+1},b)\ar[r]&
    (\Omega^n(I),b)
  }
  \]
  are pure quasi-isomorphisms.  If~\(E\) is unital, then the unital extension of the embedding \(I\to E\) induces a pure quasi-isomorphism \(\HHchain_*(I)\to \HHchain_*(E,I)\).  Thus \(\HH_*(I)\to \HH_*(E,I)\) is invertible provided~\(\Unit\) is projective.
\end{corollary}

Recall that \(\Omega^n(I)=I^+\otimes I^{\otimes n}\) for \(n\ge1\) and \(\Omega^0(I)=I\).

\begin{proof}
  The pure quasi-isomorphism \(\HHchain_*(I,I) \sim \HHchain_*(E,I)\) follows from Theorem~\ref{the:excision_ideal_extension} because~\(I\) is homologically unital if and only if it is homologically unitary as a right module over itself.  The split extension \(I\mono I^+\epi\Unit\) of modules induces a canonical split extension of chain complexes
  \[
  \HHchain_*(I,I) \mono \HHchain_*(I) \epi (I^{\otimes n},b')_{n\ge1}[1].
  \]
  Since~\(I\) is homologically unital, the chain complex \((I^{\otimes n},b')\) is pure exact.  Hence the map \(\HHchain_*(I,I) \mono \HHchain_*(I)\) is a pure quasi-isomorphism by Lemma~\ref{lem:quasi-iso_criterion}.
\end{proof}

\begin{theorem}
  \label{the:excision_extension_quotient}
  Let \(I\mono E\epi Q\) be a \emph{pure} algebra conflation, let~\(M\) be a \(Q\)\nb-bimodule.  Then we may view~\(M\) as an \(E\)\nb-bimodule.  If~\(I\) is homologically unital, then the canonical map \(\HHchain_*(E,M) \to \HHchain_*(Q,M)\) is a pure quasi-isomorphism.  Thus \(\HH_*(E,M\otimes V) \cong \HH_*(Q,M\otimes V)\) provided~\(\Unit\) is projective.
\end{theorem}

\begin{proof}
  This is the analogue of \cite{Guccione-Guccione:Excision}*{Corollary 3}.  Let~\(\tilde{F}_p\) for \(p\ge0\) be the chain complex
  \begin{multline*}
    \dotsb \leftarrow 0 \leftarrow M\otimes V
    \xleftarrow{b} M\otimes Q\otimes V
    \xleftarrow{b} M\otimes Q\otimes Q\otimes V
    \leftarrow \dotsb
    \xleftarrow{b} M\otimes Q^{\otimes p}\otimes V\\
    \xleftarrow{b} M\otimes Q^{\otimes p} \otimes E\otimes V
    \xleftarrow{b} M\otimes Q^{\otimes p} \otimes E \otimes E\otimes V
    \leftarrow \dotsb
  \end{multline*}
  with \(M\otimes V\) in degree~\(0\).  One summand in~\(b\) uses the obvious right \(E\)\nb-module structure \(E\otimes Q\to Q\) on~\(Q\).

  Since the conflation \(I\mono E\epi Q\) is pure, the map \(\tilde{F}_p \to \tilde{F}_{p+1}\) induced by the deflation \(E\epi Q\) is a deflation for each \(p\in\N\).  Its kernel is
  \[
  \ker (\tilde{F}_p\to \tilde{F}_{p+1})
  \cong M \otimes Q^{\otimes p} \otimes
  (I \otimes E^{\otimes k},b')_{k\ge0}[p+1]\otimes V.
  \]
  Since~\(I\) is homologically unital, Theorem~\ref{the:excision_ideal_extension} implies that \((I\otimes E^{\otimes k},b')\) is pure exact.  Hence the map \(\tilde{F}_p\to \tilde{F}_{p+1}\) is a pure quasi-isomorphism by Lemma~\ref{lem:quasi-iso_criterion}.  Hence so is the map \(\tilde{F}_0\to\tilde{F}_p\) for any \(p\in\N\).  This yields the assertion because \(\tilde{F}_0= \HHchain_*(E,M)\otimes V\) and \((\tilde{F}_p)_n=\HHchain_n(Q,M)\otimes V\) in degree~\(n\) for \(p\ge n\).
\end{proof}

\begin{theorem}
  \label{the:HH_excision_coefficients}
  Let \(I\mono E\epi Q\) be a pure conflation of algebras in~\(\Cat\) and assume that~\(I\) is homologically unital.  Let \(M_I\mono M_E\epi M_Q\) be a pure conflation of \(E\)\nb-modules.  Assume that the \(E\)\nb-module structure on~\(M_Q\) descends to a \(Q\)\nb-module structure and that~\(M_I\) is homologically unitary as an \(I\)\nb-module.  Then \(\HHchain_*(I,M_I) \to \HHchain_*(E,M_E)\to \HHchain_*(Q,M_Q)\) is a pure cofibre sequence.

  If~\(\Unit\) is projective, then this yields a natural long exact sequence
  \begin{multline*}
  \dotsb \to \HH_n(I,M_I) \to \HH_n(E,M_E) \to \HH_n(Q,M_Q) \\
  \to \HH_{n-1}(I,M_I) \to \HH_{n-1}(E,M_E) \to \HH_{n-1}(Q,M_Q) \to \dotsb.
  \end{multline*}
\end{theorem}

\begin{proof}
  The canonical map \(\HHchain_*(I,M_I) \to \HHchain_*(E,M_I)\) is a pure quasi-isomorphism by Theorem~\ref{the:excision_ideal_extension} because~\(M_I\) is homologically unitary as an \(I\)\nb-module.  The canonical map \(\HHchain_*(E,M_Q) \to \HHchain_*(Q,M_Q)\) is a pure quasi-isomorphism by Theorem~\ref{the:excision_extension_quotient} because~\(I\) is homologically unital.  The sequence \(\HHchain_*(E,M_I) \mono \HHchain_*(E,M_E) \to \HHchain_*(E,M_Q)\) is a pure conflation of chain complexes and thus a pure cofibre sequence because the conflation \(M_I\mono M_E\epi M_Q\) is pure.  Hence \(\HHchain_*(I,M_I) \to \HHchain_*(E,M_E) \to \HHchain_*(Q,M_Q)\) is a pure cofibre sequence as well.  If~\(\Unit\) is projective, that is, the canonical forgetful functor maps conflations to exact sequences, then this cofibre sequence implies a long exact homology sequence.
\end{proof}

\begin{theorem}
  \label{the:HH_excision}
  Let \(I\mono E\epi Q\) be a pure conflation of algebras in~\(\Cat\).  Assume that~\(I\) is homologically unital.  Then \(\HHchain_*(I) \to \HHchain_*(E)\to \HHchain_*(Q)\) is a pure cofibre sequence.  If~\(\Unit\) is projective in~\(\Cat\), this yields a natural long exact sequence
  \[
  \dotsb \to \HH_n(I) \to \HH_n(E) \to \HH_n(Q) \to \HH_{n-1}(I) \to \HH_{n-1}(E) \to \dotsb.
  \]
  If~\(\Unit\) is injective in~\(\Cat\), then there is a natural long exact sequence
  \[
  \dotsb \to \HH^n(I) \to \HH^n(E) \to \HH^n(Q) \to \HH^{n+1}(I) \to \HH^{n+1}(E) \to \dotsb.
  \]
\end{theorem}

\begin{proof}
  Let \(E^+\) and~\(Q^+\) be the algebras obtained from \(E\) and~\(Q\) by adjoining unit elements.  The algebra conflation \(I\mono E^+\epi Q^+\) is also a conflation of modules, and it is still pure because it is the direct sum of the pure conflation \(I\mono E\epi Q\) and the split extension \(0\to \Unit\xrightarrow{=} \Unit\).  Hence Theorem~\ref{the:HH_excision_coefficients} applies and yields a pure cofibre sequence \(\HHchain_*(I,I) \to \HHchain_*(E,E^+) \to \HHchain_*(Q,Q^+)\).  By definition, \(\HHchain_*(A)\oplus\Unit = \HHchain_*(A,A^+)\) for \(A\in\{E,Q\}\).  Cancelling two copies of~\(\Unit\), we get a pure cofibre sequence \(\HHchain_*(I,I) \to \HHchain_*(E) \to \HHchain_*(Q)\).  Finally, Corollary~\ref{cor:HH_II_EI} yields a pure quasi-isomorphism \(\HHchain_*(I,I) \to \HHchain_*(I)\), so that we get a pure cofibre sequence \(\HHchain_*(I) \to \HHchain_*(E) \to \HHchain_*(Q)\).

  The projectivity or injectivity of~\(\Unit\) ensures that we preserve the cofibre sequence when we apply the canonical forgetful functor or the dual space functor.  Finally, this cofibre sequence of (co)chain complexes yields the asserted long exact sequences in Hochschild homology and cohomology.
\end{proof}

Theorem \ref{the:HH_excision} is our abstract Excision Theorem for Hochschild homology and cohomology.  We can specialise it to various exact symmetric monoidal categories.

For the Abelian category~\(\Ab\), we get Wodzicki's original Excision Theorem for pure ring extensions with H\nb-unital kernel.  Our notions of purity and H\nb-unitality are the familiar ones in this case.  The dual space functor is not exact, so that we do not get assertions in cohomology.

For the Abelian category of vector spaces over some field instead of~\(\Ab\), any extension is pure and the dual space functor is exact.  Hence Hochschild homology and cohomology satisfy excision for all extensions with homologically unital kernel, and the latter means simply that \((I^{\otimes n},b')\) is exact.

Now consider the quasi-Abelian category of Fr\'echet spaces (with all extensions as conflations).  Purity means that \(I\prot V\to E\prot V\to Q\prot V\) is an extension of Fr\'echet spaces or, equivalently, an extension of vector spaces, for each Fr\'echet space~\(V\).  This is automatic if~\(Q\) is nuclear by Theorem~\ref{the:nuclear_Frechet_locally_split_extension}.  Furthermore, split extensions are pure for trivial reasons.  The dual space functor is exact by the Hahn--Banach Theorem, so that we get excision results both for Hochschild homology and cohomology.  H\nb-unitality of~\(I\) means that the chain complex \((I^{\prot n},b')_{n\ge1}\prot V\) is exact for each Fr\'echet space~\(V\), and exactness is equivalent to the vanishing of homology.  Furthermore, Theorem~\ref{the:nuclear_Frechet_locally_split_extension} shows that a nuclear Fr\'echet algebra~\(I\) is homologically unital if and only if the homology of the chain complex \((I^{\prot n},b')\) vanishes.

Let~\(\Cat\) be an additive symmetric monoidal category in which all idempotent morphisms split.  Turn~\(\Cat\) into an exact category using only the split extensions~\(\Split\).  Then any object of~\(\Cat\) is both projective and injective, and any conflation is pure because~\(\otimes\) is additive.  H\nb-unitality means that the chain complex \((I^{\otimes n},b')\) is contractible.  Thus the Excision Theorem applies to a split extension provided \((I^{\otimes n},b')\) is contractible.  The conclusion is that the map \(\HHchain_*(I) \to \cone\bigl(\HHchain_*(E)\to\HHchain_*(Q)\bigr)\) is a chain homotopy equivalence.

In the application to Whitney functions, we would like to compute \(\HH^*(Q,Q)\) by homological computations with \(E\)\nb-modules.  This is only possible under an additional injectivity assumption:

\begin{theorem}
  \label{the:excision_extension_quotient_cohomology}
  Let \(I\mono E\epi Q\) be a pure algebra conflation, let~\(M\) be a \(Q\)\nb-bimodule, which we also view as an \(E\)\nb-bimodule.  Assume that~\(I\) is homologically unital and that~\(M\) is injective as an object of~\(\Cat\).  Then the canonical map \(\HHchain^*(E,M) \to \HHchain^*(Q,M)\) is a quasi-isomorphism, so that \(\HH^*(E,M) \cong \HH^*(Q,M)\).
\end{theorem}

\begin{proof}
  Let~\(\tilde{F}_0\) for \(p\ge0\) be the cochain complex
  \begin{multline*}
    \Hom(\Unit, M) \xrightarrow{b^*} \Hom(Q,M)
    \xrightarrow{b} \Hom(Q\otimes Q,M)
    \to \dotsb
    \xrightarrow{b^*} \Hom(Q^{\otimes p},M)\\
    \xrightarrow{b^*} \Hom(Q^{\otimes p} \otimes E,M)
    \xrightarrow{b^*} \Hom(Q^{\otimes p} \otimes E \otimes E,M)
    \xrightarrow{b^*} \Hom(Q^{\otimes p} \otimes E^{\otimes 3},M)
    \to \dotsb
  \end{multline*}
  where~\(b^*\) is the Hochschild coboundary map that uses the bimodule structure on~\(M\) and the obvious right \(E\)\nb-module structure \(Q\otimes E\to Q\) on~\(Q\).  Since our algebra conflation is pure and~\(M\) is injective as an object of~\(\Cat\), we get an exact sequence of chain complexes
  \begin{equation}
    \label{eq:isocoh_quotient}
    \tilde{F}_p \mono \tilde{F}_{p+1} \epi \bigl(\Hom(Q^{\otimes p}\otimes I \otimes E^{\otimes n}[p+1],M),(b')^*\bigr).
  \end{equation}
  Theorem~\ref{the:excision_ideal_extension} implies that the chain complex \(V\otimes (I\otimes E^{\otimes k},b')\) is exact for any~\(V\) because~\(I\) is homologically unital.  Since~\(M\) is injective, the quotient complex in~\eqref{eq:isocoh_quotient} is exact.  Hence the map \(\tilde{F}_p\to \tilde{F}_{p+1}\) is a quasi-isomorphism by Lemma~\ref{lem:quasi-iso_criterion}.  Then so is the map \(\tilde{F}_0\to\tilde{F}_p\) for any \(p\in\N\).  This yields the assertion because \(\tilde{F}_0= \HHchain^*(E,M)\) and \(\tilde{F}_p^n=\HHchain^n(Q,M)\) for \(p\ge n\).
\end{proof}

Since there are few injective Fr\'echet spaces, this theorem rarely applies to the category of Fr\'echet spaces with all extensions as conflations.  One example of an injective nuclear Fr\'echet space is \(\prod_{n\in\N} \C\), the space of Whitney functions on a discrete subset of a smooth manifold.  The Schwartz space, which is isomorphic to \(\Cinf(X)\) for a non-discrete compact manifold~\(X\) and to \(\Flat{Y}{X}\) for a proper closed subset of a compact manifold~\(X\), is not injective.

A more careful choice of the conflations improves the situation.  By definition, Banach spaces are injective for locally split extensions in \(\Proban\) and hence for pro-locally split extensions of Fr\'echet spaces.  This will later allow us to do some Hochschild cohomology computations with Banach space coefficients for algebras of Whitney functions.

If we restrict to split extensions, then all objects of~\(\Cat\) become injective, so that we get the following result:

\begin{corollary}
  \label{cor:weak_excision_HH_cohomology_split}
  Let~\(\Cat\) be an additive symmetric monoidal category, equip it with the class of split extensions.  Let \(I\mono E\epi Q\) be a split extension in~\(\Cat\) and let~\(M\) be a \(Q\)\nb-bimodule.  If \((I^{\otimes n},b')\) is split exact, then \(\HH^*(E,M) \cong \HH^*(Q,M)\).
\end{corollary}

\begin{proof}
  Here any object of~\(\Cat\) is injective and any conflation is pure.  The assumption means that~\(I\) is H\nb-unital.  Hence the assertion follows from Theorem~\ref{the:excision_extension_quotient}.
\end{proof}

\section{Hochschild homology for algebras of smooth functions}
\label{sec:HH_smooth}

In this section, we work in the symmetric monoidal category \(\Frech\) of Fr\'echet spaces.  Thus \(\otimes = \prot\) is the complete projective topological tensor product of Fr\'echet spaces.  The resulting Hochschild homology and cohomology are the \emph{continuous} Hochschild homology and cohomology of Fr\'echet algebras.  We let all extensions be conflations unless we explicitly require another exact category structure on~\(\Frech\).

Let~\(X\) be a smooth manifold, possibly non-compact, and let \(\Cinf(X)\) be the Fr\'echet algebra of smooth functions on~\(X\) with the topology of locally uniform convergence of all derivatives.  The \(k\)th continuous Hochschild homology of \(\Cinf(X)\) is the space \(\Omega^k(X)\) of smooth differential \(k\)\nb-forms on~\(X\); by definition, this is the space of smooth sections of the vector bundle \((\Lambda^k \Tang X)^* \cong \Lambda^k (\Tang^*X)\).  The continuous Hochschild cohomology of \(\Cinf(X)\) is the topological dual space of continuous linear functionals on \(\Omega^k(X)\).  By definition, this is the space of distributional sections of the vector bundle \(\Lambda^k(\Tang X)\), called \emph{de Rham currents} of dimension~\(k\).

The ``continuity'' of the Hochschild homology and cohomology means that we work in the symmetric monoidal category \(\Frech\) with the tensor product~\(\prot\).  Thus \(\HHchain_*\bigl(\Cinf(X)\bigr)\defeq (\Cinf(X)^{\otimes n},b)\) involves the completed tensor products
\[
\Cinf(X)^{\otimes n} \defeq \Cinf(X)^{\prot n} \cong \Cinf(X^n).
\]
The continuous linear functionals on \(\Cinf(X)^{\prot n}\) correspond bijectively to (jointly) continuous \(n\)\nb-linear functionals \(\Cinf(X)^n\to\C\) by the universal property of~\(\prot\); hence we may describe continuous Hochschild cohomology without~\(\prot\) (as in~\cite{Connes:Noncommutative_Diffgeo}).

The continuous Hochschild cohomology of \(\Cinf(X)\) was computed by Alain Connes in~\cite{Connes:Noncommutative_Diffgeo}*{Section II.6} to prepare for the computation of its cyclic and periodic cyclic cohomology; his argument can also be used to compute the continuous Hochschild homology of \(\Cinf(X)\).  Several later argument by Jean-Luc Brylinski and Victor Nistor~\cite{Brylinski-Nistor:Cyclic_etale} and by Nicolae Teleman~\cite{Teleman:Localization_Hochschild} use localisation near the diagonal to compute the Hochschild homology and cohomology of \(\Cinf(X)\).  This localisation approach is more conceptual but, as it seems, gives slightly less information.

The chain complex \(\HHchain_*(A) = (A^{\otimes n},b)\) is a chain complex of \(A\)\nb-modules via
\[
a_0 \cdot (a_1 \otimes\dotsb \otimes a_n) \defeq
(a_0 \cdot a_1) \otimes\dotsb \otimes a_n
\]
for \(a_0,\dotsc,a_n\in A\) -- notice that this defines a chain map, that is, \(a_0\cdot b(\omega) = b(a_0\cdot\omega)\) if and only if~\(A\) is commutative.  Thus \(\HH_*\bigl(\Cinf(X)\bigr)\) inherits such a module structure as well.  The isomorphism \(\HH_*\bigl(\Cinf(X)\bigr) \cong \Omega^*(X)\) identifies this module structure on \(\HH_*\bigl(\Cinf(X)\bigr)\) with the obvious module structure on differential forms by pointwise multiplication.  We will need an even stronger result:

\begin{theorem}
  \label{the:HH_Cinf}
  Let~\(X\) be a smooth manifold.  The anti-symmetrisation map
  \begin{multline*}
    j\colon \Omega^k(X) \to \Cinf(X)^{\otimes k+1},\\    
    f_0 \,\diff f_1\wedge\dotsb\wedge \diff f_k \mapsto
    \sum_{\sigma\in \Sym{k}} (-1)^{\abs{\sigma}}
    f_0 \otimes f_{\sigma(1)}\otimes\dotsb\otimes f_{\sigma(k)}
  \end{multline*}
  and the map
  \[
  k\colon \Cinf(X)^{\otimes k+1} \to  \Omega^k(X),\qquad
  f_0 \otimes f_1\otimes\dotsb\otimes f_k \mapsto
  \frac{1}{k!}f_0 \,\diff f_1\wedge\dotsb\wedge \diff f_k
  \]
  are \(\Cinf(X)\)-linear continuous chain maps between \(\Omega^*(X)\) with the zero boundary map and \(\HHchain_*\bigl(\Cinf(X)\bigr) \defeq (\Cinf(X)^{\otimes n+1},b)\) that are inverse to each other up to \(\Cinf(X)\)-linear continuous chain homotopy.  More precisely, \(k\circ j=\Id_{\Omega^*(X)}\) and \(j\circ k = [b,h]\) for a \(\Cinf(X)\)-linear continuous map~\(h\) on~\(\HHchain_*\bigl(\Cinf(X)\bigr)\) of degree~\(1\).
\end{theorem}

\begin{proof}
  The commutativity of \(\Cinf(X)\) and the Leibniz rule \(\diff(f_1f_2) = f_1\,\diff f_2 + f_2\,\diff f_1\) in \(\Omega^*(X)\) imply \(b\circ j=0\) and \(k\circ b=0\), that is, \(j\) and~\(k\) are chain maps.  The equation \(k\circ j = \Id_{\Omega^*(X)}\) is obvious.  The only assertion that requires work is to find the \(\Cinf(X)\)-linear chain homotopy~\(h\).  The existence of such a chain homotopy follows easily from Connes' argument in~\cite{Connes:Noncommutative_Diffgeo}.  The following remarks provide some more details for readers who do not accept this one sentence as a proof.

  First we recall how the Hochschild chain complex for a Fr\'echet algebra~\(A\) is related to projective resolutions.  We must explain what ``projective resolution'' means.  The following discussion applies to any algebra~\(A\) in a symmetric monoidal category~\(\Cat\).  We call an extension of \(A\)\nb-bimodules \emph{semi-split} if it splits in~\(\Cat\) (but the splitting need not be \(A\)\nb-linear).  The semi-split extensions are the conflations of an exact category structure on the category of \(A\)-bimodules.  For any object~\(V\) of~\(\Cat\), we equip \(A\otimes V\otimes A\) with the obvious \(A\)\nb-bimodule structure and call such \(A\)\nb-bimodules \emph{free}.  Free bimodules are projective with respect to semi-split bimodule extensions.  The bar resolution \((A^{\otimes n+2},b')_{n\ge0}\) is contractible in~\(\Cat\) and hence a projective resolution of~\(A\) in the exact category of \(A\)\nb-bimodules with semi-split extensions as conflations.  Hence it is chain homotopy equivalent to any other projective \(A\)\nb-bimodule resolution of~\(A\) in this exact category.

  The \emph{commutator quotient} of an \(A\)\nb-bimodule is the cokernel of the commutator map \(A\otimes M\to M\), \(a\otimes m\mapsto [a,m] \defeq a\cdot m-m\cdot a\).  Since the commutator quotient of a free module \(A\otimes V\otimes A\) is naturally isomorphic to \(A\otimes V\), the commutator quotient complex of the bar resolution is the Hochschild chain complex \((A^{\otimes n+1},b)\).  If the algebra~\(A\) is commutative, then the commutator quotient of an \(A\)\nb-bimodule is still an \(A\)\nb-module in a canonical way.  Thus the Hochschild complex is a chain complex of \(A\)\nb-modules in a canonical way.

  If~\(P_\bullet\) is another projective \(A\)\nb-bimodule resolution of~\(A\), then~\(P_\bullet\) is chain homotopy equivalent to the bar resolution as a chain complex of \(A\)\nb-bimodules.  Hence the Hochschild chain complex is chain homotopy equivalent to the commutator quotient complex \(P_\bullet/[P_\bullet,A]\).  If~\(A\) is commutative, then this chain homotopy is \(A\)\nb-linear because the \(A\)\nb-module structure on commutator quotients is natural.

  Now we return to the Fr\'echet algebra \(\Cinf(X)\).  Connes computes the Hochschild cohomology of \(\Cinf(X)\) by constructing another projective \(\Cinf(X)\)-bimodule resolution~\(P_\bullet\) of \(\Cinf(X)\) for which the commutator quotient complex \(P_\bullet/[P_\bullet,\Cinf(X)]\) is \(\Omega^k(X)\) with zero boundary map.  As our discussion above shows, this implies that \((\Omega^*(X),0)\) is chain homotopy equivalent to \(\HHchain_*\bigl(\Cinf(X)\bigr)\) as a chain complex of \(\Cinf(X)\)-modules.  An inspection of Connes' argument also shows that the chain maps involved in this homotopy equivalence are \(j\) and~\(k\).

  More precisely, Connes' construction only applies if~\(X\) carries a nowhere vanishing vector field or, equivalently, if each connected component of~\(X\) is either non-compact or has vanishing Euler characteristic.  The case of a general smooth manifold is reduced to this case by considering \(X\times\Sphere^1\), which does carry such a vector field, and then relating the Hochschild cohomology of \(\Cinf(X)\) and \(\Cinf(X\times\Sphere^1)\).  The functoriality of Hochschild cohomology implies that \(\HHchain_*\bigl(\Cinf(X)\bigr)\) is isomorphic to the range of the map on \(\HHchain_*\bigl(\Cinf(X\times\Sphere^1)\bigr)\) induced by the map \(X\times\Sphere^1\to X\times\Sphere^1\), \((x,z)\mapsto (x,1)\).  Under the homotopy equivalence between \(\HHchain_*\bigl(\Cinf(X\times\Sphere^1)\bigr)\) and \((\Omega^*(X\times\Sphere^1),0)\), this map corresponds to a projection onto \((\Omega^*(X),0)\).
\end{proof}

The additional statements about chain homotopy equivalence in Theorem~\ref{the:HH_Cinf} seem difficult to prove with the localisation method because the latter involves contractible subcomplexes that are either not even closed (such as the chain complex of functions vanishing in some neighbourhood of the diagonal) or are not complementable (such as the chain complex of functions that are flat on the diagonal).

\subsection{The algebra of smooth functions with compact support}
\label{sec:smooth_compact_support}

Now we want to replace the Fr\'echet algebra \(\Cinf(X)\) by the dense subalgebra \(\Ccinf(X)\) of smooth functions with compact support.  This is an LF-space in a natural topology: Let~\((K_n)_{n\in\N}\) be an increasing sequence of compact subsets exhausting~\(X\), then \(\Ccinf(X)\) is the strict inductive limit of the subspaces of \(\Cinf(X)\) of smooth functions that vanish outside~\(K_n\).  This is a topological algebra, that is, the multiplication is jointly continuous.  Nevertheless, we will view \(\Ccinf(X)\) as a bornological algebra in the following, that is, replace it by \(\Prec \Ccinf(X)\).  This is preferable because the projective bornological tensor product agrees with Grothendieck's inductive tensor product for nuclear LF-spaces, so that \(\Prec \Ccinf(X) \hot \Prec \Ccinf(Y) \cong \Prec \Ccinf(X\times Y)\) for all smooth manifolds \(X\) and~\(Y\).  In contrast, the projective topological tensor product is \(\Ccinf(X\times Y)\) with a complicated topology.  Since we want tensor powers of \(\Ccinf(X)\) to be \(\Ccinf(X^n)\), we must either define tensor products in an \emph{ad hoc} way as in~\cite{Brodzki-Plymen:Periodic} or work bornologically.

We turn the category of complete bornological \(\Ccinf(X)\)-modules into an exact category using the class of split extensions as conflations, that is, conflations are extensions of \(\Ccinf(X)\)-modules with a bounded linear section.

We already have a projective bimodule resolution for \(\Cinf(X)\) and want to use it to construct one for \(\Ccinf(X)\).  Given a \(\Ccinf(X)\)-module~\(M\), we let \(M_\cpt\subseteq M\) be the subspace of all \(m\in M\) for which there is \(f\in\Ccinf(X)\) with \(m=f\cdot m\).  This agrees with \(\Ccinf(X)\cdot M\) because \(\Cinf(X)_\cpt = \Ccinf(X)\).  A subset~\(S\) of~\(M_\cpt\) is called \emph{bounded} if it is bounded in~\(M\) and there is a single \(f\in\Ccinf(X)\) with \(f\cdot m = m\) for all \(m\in S\).  This defines a complete bornology on~\(M_\cpt\).  The subspace~\(M_\cpt\) is still a module over \(\Cinf(X)\) and, \emph{a fortiori}, over~\(\Ccinf(X)\).  The multiplication maps \(\Ccinf(X)\times M_\cpt\to M_\cpt\) and \(\Cinf(X)\times M_\cpt\to M_\cpt\) are both bounded.

\begin{proposition}
  \label{pro:compact_support_functor_exact}
  The functor \(M\mapsto M_\cpt\) is exact.  If~\(M\) is a projective \(\Cinf(X)\)-module, then~\(M_\cpt\) is projective both as a \(\Ccinf(X)\)-modules and as a \(\Cinf(X)\)-module.
\end{proposition}

\begin{proof}
  Exactness means that \(I_\cpt \mono E_\cpt \epi Q_\cpt\) is a semi-split extension if \(I\mono E\epi Q\) is a semi-split extension of \(\Cinf(X)\)-modules.  Let \(s\colon Q\to E\) be a bounded linear section.  Let \((\varphi_n)_{n\in\N}\) be a locally finite set of compactly supported functions with \(\sum \varphi_n^2 = 1\), that is, \(\varphi_n^2\) is a partition of unity.  We define \(s_\cpt(f) \defeq \sum_{n\in\N} \varphi_n\cdot s(\varphi_n\cdot f)\).  This is still a well-defined bounded linear section \(Q\to E\), but this new section restricts to a bounded linear map \(Q_\cpt\to E_\cpt\) because for each compact subset~\(K\) we have \(\varphi_n|_K=0\) for almost all~\(n\) and each~\(\varphi_n\) has compact support.  As a consequence, the functor \(M\mapsto M_\cpt\) is exact.

  Since any projective \(\Cinf(X)\)-module is a direct summand of a free module \(\Cinf(X)\hot V\), the claim about projectivity means that \((\Cinf(X)\hot V)_\cpt\) is projective as a \(\Cinf(X)\)-module and as a \(\Ccinf(X)\)-module for any~\(V\).  Since \((\Cinf(X)\hot V)_\cpt \cong \Ccinf(X)\hot V\) and \(M\hot V\) is projective once~\(M\) is projective, we must check that \(\Ccinf(X)\) is projective as a \(\Cinf(X)\)-module and as a \(\Ccinf(X)\)-module.  Equivalently, the multiplication maps \(\Cinf(X)\hot \Ccinf(X) \to \Ccinf(X)\) and \(\Ccinf(X)^+\hot \Ccinf(X)\to \Ccinf(X)\) split by a module homomorphism.  This follows from the following lemma, which finishes the proof.
\end{proof}

\begin{lemma}
  \label{lem:Ccinf_projective}
  The multiplication map \(\Ccinf(X)\hot \Ccinf(X)\to \Ccinf(X)\) has bounded linear sections \(\sigma_\leftsub\) and~\(\sigma_\rightsub\) that are a left and a right module homomorphism, respectively.
\end{lemma}

\begin{proof}
  Recall that \(\Ccinf(X)\hot \Ccinf(X)\cong \Ccinf(X\times X)\).  The multiplication map becomes the map \(\Ccinf(X\times X)\to\Ccinf(X)\) that restricts functions to the diagonal.  It suffices to describe~\(\sigma_\leftsub\), then \((\sigma_\rightsub f)(x,y) \defeq (\sigma_\leftsub f)(y,x)\) provides~\(\sigma_\rightsub\).  We make the Ansatz \((\sigma_\leftsub f)(x,y) \defeq f(x)\cdot w(x,y)\) for some \(w\in\Cinf(X\times X)\).  We need \(w(x,x)=1\) for all \(x\in X\) in order to get a section for the multiplication map, and we assume that the projection to the first coordinate \((x,y)\mapsto x\) restricts to a proper map on the support of~\(w\).  That is, for each compact subset \(K\subseteq X\) there is a compact subset \(L\subseteq X\times X\) such that \(w(x,y)=0\) for all \(x\in K\) with \((x,y)\notin L\).  This ensures that \(\sigma_\leftsub f\) is supported in~\(L\) if~\(f\) is supported in~\(K\).  It is routine to check that such a function~\(w\) exists and that the resulting map~\(\sigma_\leftsub\) has the required properties.
\end{proof}

Lemma~\ref{lem:Ccinf_projective} implies that \(\Ccinf(X)\) is H\nb-unital, see the proof of Corollary~\ref{cor:Flat_H-unital} below.

Viewing \(\Ccinf(X)\)-bimodules as \(\Ccinf(X\times X)\)-modules, the above carries over to bimodules without change.  Now the support restriction functor \(M\mapsto M_\cpt\) will require compact support with respect to both module structures.  Since \(\Cinf(X)_\cpt=\Ccinf(X)\), Proposition~\ref{pro:compact_support_functor_exact} implies that the functor \(P\mapsto P_\cpt\) maps any projective \(\Cinf(X)\)-bimodule resolution of \(\Cinf(X)\) to a projective \(\Ccinf(X)\)-bimodule resolution of \(\Ccinf(X)\).

It is easy to check that the commutator quotient functor for bimodules intertwines the support restriction functors:
\[
\bigl(M \mathbin/ [\Cinf(X),M]\bigr)_\cpt \cong
M_\cpt \bigm/ [\Ccinf(X),M_\cpt].
\]
As a consequence, \(\HHchain_*\bigl(\Ccinf(X)\bigr)\) is chain homotopy equivalent to \(\HHchain_*\bigl(\Cinf(X)\bigr)_\cpt\), where the support restriction is with respect to the canonical \(\Cinf(X)\)-module structure on \(\HHchain_*\bigl(\Cinf(X)\bigr)_\cpt\).  Theorem~\ref{the:HH_Cinf} now implies:

\begin{theorem}
  \label{the:HH_Ccinf}
  Let~\(X\) be a smooth manifold.  The anti-symmetrisation map
  \begin{multline*}
    j\colon \Omega_\cpt^k(X) \to \Ccinf(X)^{\otimes k+1},\\
    f_0 \,\diff f_1\wedge\dotsb\wedge \diff f_k \mapsto
    \sum_{\sigma\in \Sym{k}} (-1)^{\abs{\sigma}}
    f_0 \otimes f_{\sigma(1)}\otimes\dotsb\otimes f_{\sigma(k)}
  \end{multline*}
  and the map
  \[
  k\colon \Ccinf(X)^{\otimes k+1} \to  \Omega_\cpt^k(X),\\
  f_0 \otimes f_1\otimes\dotsb\otimes f_k \mapsto
  \frac{1}{k!}f_0 \,\diff f_1\wedge\dotsb\wedge \diff f_k
  \]
  are \(\Ccinf(X)\)-linear bounded chain maps between \(\Omega_\cpt^*(X)\) with the zero boundary map and \(\HHchain_*\bigl(\Ccinf(X)\bigr) \defeq (\Ccinf(X)^{\otimes n+1},b)\) that are inverse to each other up to \(\Ccinf(X)\)-linear bounded chain homotopy.  More precisely, \(k\circ j=\Id_{\Omega_\cpt^*(X)}\) and \(j\circ k = [b,h]\) for an \(\Ccinf(X)\)-linear bounded map~\(h\) on~\(\HHchain_*\bigl(\Ccinf(X)\bigr)\) of degree~\(1\).
\end{theorem}

Here \(\Omega_\cpt^k(X) \defeq \Omega^k(X)_\cpt\) is the space of compactly supported smooth \(k\)\nb-forms with its canonical bornology.

\section{Application to Whitney functions}
\label{sec:Whitney}

As in the previous section, we work in the symmetric monoidal category~\(\Frech\) of Fr\'echet spaces.

Let~\(X\) be a smooth manifold --~for instance, an open subset of~\(\R^n\)~-- and let~\(Y\) be a closed subset of~\(X\).  A smooth function on~\(X\) is called \emph{flat} on~\(Y\) if its Taylor series vanishes at each point of~\(Y\), that is, \((Df)(y)=0\) for any \(y\in Y\) and any differential operator~\(D\), including operators of order zero.  The smooth functions that are flat on~\(Y\) form a closed ideal~\(\Flat{Y}{X}\) in~\(\Cinf(X)\).  The quotient
\[
\Whitney(Y) \defeq \Cinf(X) \bigm/ \Flat{Y}{X}
\]
is, by definition, the Fr\'echet algebra of \emph{Whitney functions} on~\(Y\).  This algebra depends on the embedding of~\(Y\) in~\(X\).

\begin{example}
  \label{exa:Whitney_point}
  If \(Y\subseteq X\) consists of a single point, then \(\Whitney(Y)\) is isomorphic to the Fr\'echet algebra \(\C[\![x_1,\dotsc,x_n]\!]\) of formal power series in \(n\)~variables, where~\(n\) is the dimension of~\(X\).
\end{example}

By definition, the Fr\'echet algebra of Whitney functions fits into an extension
\[
\Flat{Y}{X} \mono \Cinf(X) \epi \Whitney(Y),
\]
as in~\eqref{eq:Whitney_extension}.  It is well-known that \(\Cinf(X)\) is nuclear, see~\cite{Grothendieck:Produits}.  Since nuclearity is inherited by subspaces and quotients, \(\Flat{Y}{X}\) and \(\Whitney(Y)\) are nuclear as well.  Theorem~\ref{the:nuclear_Frechet_locally_split_extension} shows, therefore, that the extension~\eqref{eq:Whitney_extension} is both ind-locally split and pro-locally split, and this remains so if we tensor first with another Fr\'echet space~\(V\).  As a consequence, \eqref{eq:Whitney_extension} is a pure extension.

Our next goal is to show that \(\Flat{Y}{X}\) is homologically unital.  This requires computing some complete projective topological tensor products and finding a smooth function with certain properties.

\begin{lemma}
  \label{lem:tensor_Flat}
  There are isomorphisms
  \begin{align*}
    \Cinf(X)\prot\Cinf(X) &\cong \Cinf(X\times X),\\
    \Flat{Y}{X}\prot\Flat{Y}{X} &\cong \Flat{Y\times X\cup X\times Y}{X\times X}.
  \end{align*}
\end{lemma}

\begin{proof}
  It is well-known that \(\Cinf(X)\prot\Cinf(X) \cong \Cinf(X\times X)\), see~\cite{Grothendieck:Produits}.  Since all spaces involved are nuclear, \(\Cinf(X)\prot\Flat{Y}{X}\), \(\Flat{Y}{X}\prot\Cinf(X)\), and \(\Flat{Y}{X}\prot\Flat{Y}{X}\) are subspaces of \(\Cinf(X\times X)\) --~they are the closures of the corresponding algebraic tensor products in \(\Cinf(X\times X)\)~-- and
  \begin{equation}
    \label{eq:Flat_tensor_Flat}
    \Flat{Y}{X}\prot\Flat{Y}{X} =
    \Cinf(X)\prot\Flat{Y}{X} \cap
    \Flat{Y}{X}\prot\Cinf(X)
  \end{equation}
  as subspaces of \(\Cinf(X\times X)\).

  Since functions of the form \(f_1\otimes f_2\) with \(f_1\in\Flat{Y}{X}\), \(f_2\in\Cinf(X)\) span a dense subspace of \(\Flat{Y\times X}{X\times X}\), we get
  \[
  \Cinf(X)\prot\Flat{Y}{X} \cong \Flat{X\times Y}{ X\times X}.
  \]
  Similarly,
  \[
  \Flat{Y}{X}\prot\Cinf(X) \cong \Flat{Y\times X}{ X\times X}.
  \]
  Now~\eqref{eq:Flat_tensor_Flat} shows that a smooth function on \(X\times X\) belongs to \(\Flat{Y}{X}\prot\Flat{Y}{X}\) if and only if it is flat on both \(Y\times X\) and \(X\times Y\).
\end{proof}

\begin{proposition}
  \label{pro:Flat_quasi-unital}
  There are continuous linear sections
  \[
  \sigma_\leftsub,\sigma_\rightsub\colon \Flat{Y}{X} \to \Flat{Y}{X} \prot \Flat{Y}{X}
  \]
  for the multiplication map \(\mu\colon \Flat{Y}{X}\prot\Flat{Y}{X}\to \Flat{Y}{X}\), such that~\(\sigma_\leftsub\) is a left \(\Cinf(X)\)-module map and~\(\sigma_\rightsub\) is a right \(\Cinf(X)\)-module map.
\end{proposition}

\begin{proof}
  Let \(X_2\defeq X\times X\) and \(Y_2\defeq Y\times X\cup X\times Y\subseteq X_2\).  By Lemma~\ref{lem:tensor_Flat}, \(\sigma_\leftsub\) and~\(\sigma_\rightsub\) are supposed to be maps from \(\Flat{Y}{X}\) to \(\Flat{Y_2}{X_2}\).  The multiplication map corresponds to the map \(\mu\colon \Flat{Y_2}{X_2}\to \Flat{Y}{X}\) that restricts functions to the diagonal, \((\mu f)(x)\defeq f(x,x)\).  Once we have found~\(\sigma_\leftsub\), we may put \(\sigma_\rightsub f(x_1,x_2) \defeq \sigma_\leftsub f(x_2,x_1)\), so that it suffices to construct~\(\sigma_\leftsub\).

  Let \(A\subseteq X_2\) be the diagonal and let \(B\defeq X\times Y\).  Then \(A\cap B\) is the diagonal image of~\(Y\) in~\(X_2\).  We claim that \(A\) and~\(B\) are regularly situated (see~\cite{Tougeron:Ideaux}); even more, \(d(x,A\cap B) \le C\cdot(d(x,A)+d(x,B)\bigr)\) for some constant~\(C\).  The distance from \((x_1,x_2)\) to the diagonal~\(A\) is \(d(x_1,x_2)\); the distance to~\(B\) is \(d(x_2,Y)\); and the distance to \(A\cap B\) is at most
  \[
  \inf_{y\in Y} d(x_1,y)+d(x_2,y)
  \le \inf_{y\in Y} d(x_1,x_2)+d(x_2,y)+d(x_2,y)
  = d(x_1,x_2) + 2d(x_2,Y).
  \]

  Since \(A\) and~\(B\) are regularly situated, \cite{Tougeron:Ideaux}*{Lemma 4.5} yields a multiplier~\(w\) of \(\Flat{A\cap B}{X_2}\) that is constant equal to one in a neighbourhood of \(A\setminus (A\cap B)\) and constant equal to zero in a neighbourhood of \(B\setminus (A\cap B)\).\footnote{I thank Markus Pflaum for this construction of~\(w\).}  Being a multiplier means that~\(w\) is a smooth function on \(X_2\setminus (A\cap B)\) whose derivatives grow at most polynomially near \(A\cap B\), that is, for each compactly supported differential operator~\(D\) there is a polynomial~\(p_D\) with \(\abs{Dw(x)} \le p_D(d(x, A\cap B)^{-1})\) for all \(x\in X_2\setminus (A\cap B)\).  Now define
  \[
  \sigma_\leftsub f(x_1,x_2) \defeq f(x_1)\cdot w(x_1,x_2).
  \]
  The function~\(\sigma_\leftsub f\) is a smooth function on \(X_2\setminus (A\cap B)\).  Since \(f\otimes 1\) is flat on \(Y\times X \supseteq A\cap B\) and~\(w\) is a multiplier of \(\Flat{A\cap B}{X_2}\), its extension by~\(0\) on \(A\cap B\), also denoted by~\(\sigma_\leftsub f\), is a smooth function on~\(X_2\) that is flat on \(Y\times X\).  Furthermore, the extension is flat on \((X\setminus Y)\times Y\) because~\(w\) is, so that \(\sigma_\leftsub f\in \Flat{Y_2}{X_2}\).  We also have \(\sigma_\leftsub f(x,x) = f(x)\) both for \(x\in X\setminus Y\) and \(x\in Y\).  Thus~\(\sigma_\leftsub\) has the required properties.
\end{proof}

It can be checked that \(\Flat{Y}{X}\) has a multiplier bounded approximate unit as well; hence it is quasi-unital in the notation of~\cite{Meyer:Embed_derived}.  We do not need this stronger fact here.

\begin{corollary}
  \label{cor:Flat_H-unital}
  Let \(I\defeq \Flat{Y}{X}\), then the chain complex \((I^{\otimes n},b')\) has a bounded contracting homotopy, that is, \(I\) is homologically unital for any exact category structure on the category of Fr\'echet spaces.
\end{corollary}

\begin{proof}
  The maps \(s_n\defeq \sigma_\rightsub\otimes\Id_{I^{\otimes n-1}} \colon I^{\otimes n} \to I^{\otimes n+1}\) for \(n\ge1\) satisfy \(b' s_1 = \Id_I\) and \(b' s_n + s_{n-1} b' = \Id_{I^{\otimes n}}\) for \(n\ge2\) because~\(\sigma_\rightsub\) is a right module map and a section for the multiplication map.  Hence \((I^{\otimes n},b')\) is contractible.  So is \((I^{\otimes n},b')\otimes V\) for any Fr\'echet space~\(V\).  Contractible chain complexes are exact.
\end{proof}

Having checked that the extension~\eqref{eq:Whitney_extension} is pure and that its kernel is homologically unital, we are in a position to apply Theorem~\ref{the:HH_excision}.  Since \(\Unit=\C\) is both projective and injective as a Fr\'echet space, we get long exact sequences in Hochschild homology and cohomology for the algebra extension~\eqref{eq:Whitney_extension}.  We have already computed the Hochschild homology and cohomology of \(\Cinf(X)\) in the previous section.  Our next task is to compute them for \(\Flat{Y}{X}\), together with the maps on Hochschild homology and cohomology induced by the embedding \(\Flat{Y}{X}\to\Cinf(X)\).

The following computations also use the balanced tensor product \(M\otimes_A N\) for a right \(A\)\nb-module~\(M\) and a left \(A\)\nb-module~\(N\).  This is defined --~in the abstract setting of symmetric monoidal categories with cokernels~-- as the cokernel of the map \(b'\colon M\otimes A\otimes N\to M\otimes N\).

\begin{proposition}
  \label{pro:HH_Flat}
  The Hochschild homology \(\HH_k\bigl(\Flat{Y}{X}\bigr)\) is isomorphic to the space \(\Flat[^\infty\Omega^k]{Y}{X}\) of smooth differential \(k\)\nb-forms on~\(X\) that are flat on~\(Y\), and the map to \(\HH_k\bigl(\Cinf(X)\bigr) \cong \Omega^k(X)\) is equivalent to the obvious embedding \(\Flat[^\infty\Omega^k]{Y}{X} \to \Omega^k(X)\).  The Hochschild cohomologies are naturally isomorphic to the topological dual spaces \(\Flat[^\infty\Omega^k]{Y}{X}^*\) and \(\Omega^k(X)^*\).
\end{proposition}

\begin{proof}
  Let \(I\defeq \Flat{Y}{X}\) and \(E\defeq \Cinf(X)\).  Theorem~\ref{the:HH_Cinf} shows that the chain complex \(\HHchain_*(E) \defeq (E^{\otimes n},b)_{n\ge1}\) is homotopy equivalent as a chain complex of \(E\)\nb-modules to \(\Omega^*(X)\) with zero as boundary map.  Hence
  \[
  I\otimes_E \HHchain_*(E) \cong (I\otimes E^{\otimes n-1},b)_{n\ge1} = \HHchain_*(E,I)
  \]
  is homotopy equivalent (as a chain complex of \(I\)\nb-modules) to
  \[
  I\otimes_E \Omega^*(X) = \Flat{Y}{X} \otimes_{\Cinf(X)} \Omega^*(X)
  \cong \Flat[^\infty\Omega^*]{Y}{X}.
  \]
  The last step follows from the following more general computation.  If~\(V\) is a smooth vector bundle on~\(X\) and~\(\Gamma^\infty(V)\) is its space of smooth sections, then \(\Flat{Y}{X} \otimes_{\Cinf(X)} \Gamma^\infty(V)\) is the space of smooth sections of~\(V\) that are flat on~\(Y\).  This is easy to see for trivial~\(V\), and Swan's Theorem reduces the general case to this special case.

  Corollaries \ref{cor:Flat_H-unital} and~\ref{cor:HH_II_EI} show that the canonical embedding \(\HHchain_*(I)\to\HHchain_*(E,I)\) is a pure quasi-isomorphism, so that \(\HHchain_*(I)\) is purely quasi-isomorphic to \(I\otimes_E \HHchain_*(E)\).  Hence \(\HH_*(I) \cong \Flat[^\infty\Omega^k]{Y}{X}\) as asserted.  Furthermore, our computation shows that the map \(\HH_*(I)\to\HH_*(E)\) induced by the embedding \(I\to E\) is the obvious embedding \(\Flat[^\infty\Omega^k]{Y}{X} \to \Omega^k(X)\).

  The range of the boundary map of \(\HHchain_*(E)\) is closed.  Thus \(\HH_k(E)\) inherits a Fr\'echet topology -- this is, of course, the standard Fr\'echet space structure on \(\Omega^k(X)\).  Since the dual space functor on Fr\'echet spaces is exact, there is the following universal coefficient theorem: if the boundary map of a chain complex of Fr\'echet spaces has closed range, then the cohomology of the topological dual chain complex is the topological dual of the homology.  Thus \(\HH^k(E) \cong \Omega^k(X)^*\).  Our computation shows that the boundary map of \(\HHchain_*(I)\) has closed range as well, so that \(\HH^k(E) \cong \Flat[^\infty\Omega^k]{Y}{X}^*\).
\end{proof}

Since the embedding \(\Flat[^\infty\Omega^k]{Y}{X}\to \Omega^k(X)\) has
closed range, the quotient space
\[
\Whitney\Omega^k(Y) \defeq \Omega^k(X) \bigm/ \Flat[^\infty\Omega^k]{Y}{X}
\]
is a Fr\'echet space in the quotient topology.  This is the space of \emph{Whitney differential forms}.  Since \(\Omega^k(X)\) is a projective \(\Cinf(X)\)-module and the extension~\eqref{eq:Whitney_extension} is pure, the diagram
\[
\Flat{Y}{X} \otimes_{\Cinf(X)} \Omega^k(X) \to
\Cinf(X) \otimes_{\Cinf(X)} \Omega^k(X) \to
\Whitney(Y) \otimes_{\Cinf(X)} \Omega^k(X)
\]
is an extension of Fr\'echet spaces.  Identifying
\[
\Flat{Y}{X}\otimes_{\Cinf(X)} \Omega^k(X) \cong \Flat[^\infty\Omega^k]{Y}{X},\qquad
\Cinf(X) \otimes_{\Cinf(X)} \Omega^k(X) \cong \Omega^k(X),
\]
we see that
\[
\Whitney\Omega^k(Y) \cong \Whitney(Y) \otimes_{\Cinf(X)} \Omega^k(X).
\]
This provides an alternative definition for the space of Whitney differential forms.  The de Rham boundary map on \(\Omega^*(X)\) maps the subspace \(\Flat[^\infty\Omega^*]{Y}{X}\) into itself and hence induces a de Rham boundary map
\[
0 \to \Whitney\Omega^0(Y) \xrightarrow{\diff} \Whitney\Omega^1(Y) \xrightarrow{\diff} \Whitney\Omega^2(Y) \xrightarrow{\diff}  \Whitney\Omega^3(Y) \to \dotsb
\]
on Whitney differential forms.  The cohomology of this cochain complex is the \emph{de Rham cohomology} \(\Ho_\dR^*(Y)\) of~\(Y\).

\begin{theorem}
  \label{the:HH_HC_HP_Whitney}
  Let~\(Y\) be a closed subset of a smooth manifold~\(X\) and let \(\Cinf(Y)\) be the Fr\'echet algebra of Whitney functions on \(Y\subseteq X\).  Then
  \begin{align*}
    \HH_k\bigl(\Whitney(Y)\bigr) &\cong \Whitney\Omega^k(Y),\\
    \HC_k\bigl(\Whitney(Y)\bigr) &\cong
    \Whitney\Omega^k(Y)\bigm/ \diff(\Whitney\Omega^{k-1}Y) \oplus \Ho_\dR^{k-2}(Y)
    \oplus \Ho_\dR^{k-4}(Y) \oplus\dotsb,\\
    \HP_k\bigl(\Whitney(Y)\bigr) &\cong
    \bigoplus_{j\in\Z} \Ho_\dR^{k-2j}(Y).
  \end{align*}
\end{theorem}

\begin{proof}
  The Excision Theorem~\ref{the:HH_excision} provides a long exact sequence
  \begin{multline*}
    \dotsb \to \HH_k\bigl(\Flat{Y}{X}\bigr) \to \HH_k\bigl(\Cinf(X)\bigr) \to \HH_k\bigl(\Whitney(Y)\bigr)\\
    \to \HH_{k-1}\bigl(\Flat{Y}{X}\bigr) \to \HH_{k-1}\bigl(\Cinf(X)\bigr) \to \HH_{k-1}\bigl(\Whitney(Y)\bigr) \to \dotsb
  \end{multline*}
  because \(\Flat{Y}{X}\) is homologically unital and the extension~\eqref{eq:Whitney_extension} is pure.  We have seen above that the map \(\HH_k\bigl(\Flat{Y}{X}\bigr) \to \HH_k\bigl(\Cinf(X)\bigr)\) is equivalent to the map \(\Flat[^\infty\Omega^k]{Y}{X} \to \Omega^k(X)\).  The latter is a closed embedding and, in particular, injective.  Hence the long exact sequence above yields \(\HH_k\bigl(\Whitney(Y)\bigr) \cong \Whitney\Omega^k(Y)\) as asserted.

  To go from Hochschild homology to cyclic homology, we may use the natural map \(B\circ I\colon \HH_k(A) \to \HH_{k+1}(A)\).  For \(A=\Cinf(X)\), this map corresponds to the de Rham boundary map.  By naturality, this remains true for \(\Whitney(Y)\) as well.  The same arguments as for smooth manifolds now yield the formulas for the cyclic and periodic cyclic homology of~\(Y\).
\end{proof}

If the de Rham boundary map on \(\Whitney\Omega^*(Y)\) has closed range, then \(\HC_*\bigl(\Whitney(Y)\bigr)\) and \(\HP_*\bigl(\Whitney(Y)\bigr)\) are Fr\'echet spaces, and the cohomology spaces \(\HC^*\bigl(\Whitney(Y)\bigr)\) and \(\HP^*\bigl(\Whitney(Y)\bigr)\) are just their topological duals.  Together with Markus Pflaum, I plan to show in a forthcoming article that the boundary map on \(\Whitney\Omega^*(Y)\) always has closed range, and to identify the resulting cohomology theory with the Alexander--Spanier cohomology of~\(Y\) in favourable cases.

The results above carry over with small changes to the algebra of compactly supported Whitney functions.  There are, in fact, two ways to proceed.  Either we use Theorem \ref{the:HH_Ccinf} instead of~\ref{the:HH_Cinf} and copy the excision argument above; this requires the excision theorem for ind-locally split extensions of nuclear bornological algebras.  Or we copy the reduction from \(\Ccinf(X)\) to \(\Cinf(X)\) for the algebra of Whitney functions.  We may define a functor \(M\mapsto M_\cpt\) also for modules over \(\Whitney(Y)\), and everything said in Section~\ref{sec:smooth_compact_support} carries over to Whitney functions instead of smooth functions without change.  This shows that the Hochschild chain complex for the algebra of compactly supported Whitney functions is \(\HHchain\bigl(\Whitney(Y)\bigr)_\cpt\).  Now we can copy the arguments above.  De Rham cohomology is replaced by compactly supported de Rham cohomology, of course:

\begin{theorem}
  \label{the:HH_HC_HP_Whitney_c}
  Let~\(Y\) be a closed subset of a smooth manifold~\(X\) and let \(\Whitney_\cpt(Y)\) be the bornological algebra of compactly supported Whitney functions on \(Y\subseteq X\).  Then
  \begin{align*}
    \HH_k\bigl(\Whitney_\cpt(Y)\bigr) &\cong \Whitney_\cpt\Omega^k(Y),\\
    \HC_k\bigl(\Whitney_\cpt(Y)\bigr) &\cong
    \Whitney_\cpt\Omega^k(Y)\bigm/ \diff(\Whitney_\cpt\Omega^{k-1} Y) \oplus \Ho_{\dR,\cpt}^{k-2}(Y)
    \oplus \Ho_{\dR,\cpt}^{k-4}(Y) \oplus\dotsb,\\
    \HP_k\bigl(\Whitney_\cpt(Y)\bigr) &\cong
    \bigoplus_{j\in\Z} \Ho_{\dR,\cpt}^{k-2j}(Y).
  \end{align*}
\end{theorem}

Finally, we turn to Hochschild cohomology with coefficients.

\begin{lemma}
  \label{lem:Cinf_cohomology_symmetric}
  Let~\(M\) be a \(\Cinf(X)\)-module, viewed as a symmetric \(\Cinf(X)\)-bimodule.  Then \(\HHchain^*(\Cinf(X),M)\) is chain homotopy equivalent to \(\Hom_{\Cinf(X)}(\Omega^*(X),M)\) with zero boundary map.  In particular, \(\HH^k(\Cinf(X),M)\cong \Hom_{\Cinf(X)}(\Omega^*(X),M)\).
\end{lemma}

\begin{proof}
  For any commutative unital algebra~\(A\), the chain complex \(\HHchain^*(A,M)\) is naturally isomorphic to \(\Hom_A(\HHchain_*(A),M)\) because
  \[
  \Hom_A(A\otimes A^{\otimes n},M) \cong \Hom(A^{\otimes n},M)
  \]
  and~\(b\) on \(A\otimes A^{\otimes n}\) induces the boundary map~\(b^*\) on \(\Hom(A^{\otimes n},M)\).  Theorem~\ref{the:HH_Cinf} asserts that \(\HHchain_*\bigl(\Cinf(X)\bigr)\) is \(\Cinf(X)\)-linearly chain homotopy equivalent to \(\Omega^*(X)\) with zero boundary map.  Hence we get an induced chain homotopy equivalence between \(\HHchain^*(\Cinf(X),M)\) and \(\Hom_{\Cinf(X)}(\Omega^*(X),M)\) with zero boundary map.
\end{proof}

Let \(Q\defeq \Whitney(Y)\), \(E\defeq \Cinf(X)\), and let~\(M\) be a \(Q\)\nb-module.  We may view~\(M\) as an \(E\)\nb-module, as a symmetric \(Q\)\nb-bimodule, and as a symmetric \(E\)\nb-bimodule.  Whenever Theorem~\ref{the:excision_extension_quotient_cohomology} applies, we get
\[
\HH^*(\Whitney(Y),M) \cong \HH^*(\Cinf(X),M)
\cong \Hom_{\Cinf(X)}(\Omega^*(X),M).
\]
Unfortunately, it is not clear whether this holds in the most interesting case \(M=\Cinf(X)\).  But since the extension \(\Flat{Y}{X} \mono \Cinf(X) \epi \Whitney(Y)\) is pure pro-locally split by Theorem~\ref{the:nuclear_Frechet_locally_split_extension}, we may use the exact category structure of pro-locally split extensions and apply Theorem~\ref{the:excision_extension_quotient_cohomology} whenever~\(M\) is a Banach space.  In particular, we may use the \(k\)~times differentiable version \(\Whitney[k](Y)\) of \(\Whitney(Y)\), which is the Banach space quotient of the Banach algebra \(\Cinf[k](X)\) of \(k\)~times continuously differentiable functions on~\(X\) by the closed ideal of all \(k\)~times differentiable functions whose derivatives of order at most~\(k\) vanish on~\(Y\).

\begin{theorem}
  \label{the:Whitney_cohomology_Banach}
  \(\HH^k\bigl(\Whitney(Y),\Whitney[k](Y)\bigr)\) is isomorphic to the space of \(k\)~times continuously differentiable Whitney \(k\)\nb-vector fields on~\(Y\).
\end{theorem}

\begin{proof}
  This follows from the above discussion and a straightforward computation of \(\Hom_{\Cinf(X)}\bigl(\Omega^*(X), \Whitney[k](Y)\bigr)\) that uses Swan's Theorem.  Of course, the space of \(k\)~times continuously differentiable Whitney \(k\)\nb-vector fields on~\(Y\) is the quotient of the space of \(k\)~times continuously differentiable \(k\)\nb-vector fields on~\(X\) by the subspace of vector fields that are flat on~\(Y\) (up to order~\(k\)).
\end{proof}

\section{Excision in periodic cyclic homology}
\label{sec:excision_periodic}

Joachim Cuntz and Daniel Quillen established that periodic cyclic homology satisfies excision for all algebra extensions~\cite{Cuntz-Quillen:Excision_bivariant}, even if the kernel is not homologically unital.  This leads us to expect that periodic cyclic homology for algebras in a symmetric monoidal category should satisfy an excision theorem for all pure algebra extensions.  But I do not know how to establish excision in this generality.  Instead, I only recall two more special results already in the literature and apply them to the algebra of smooth functions on a closed subset of a smooth manifold.

Let~\(\Cat\) be a symmetric monoidal category and let~\(\Procat\) be the category of projective systems in~\(\Cat\) with the induced tensor product.  For instance, the category \(\Proban\) of projective systems of Banach spaces is a special case of this which is closely related to topological vector spaces.  We assume that~\(\Cat\) is \(\Q\)\nb-linear because otherwise homotopy invariance and excision theorems fail.  We use locally split extensions as conflations to turn~\(\Procat\) into an exact category.  That is, a diagram \(I\to E\to Q\) in~\(\Procat\) is a conflation if and only if
\[
0\leftarrow\Hom(I,V)\leftarrow \Hom(E,V) \leftarrow
\Hom(Q,V)\leftarrow0
\]
is a short exact sequence for any \(V\in\Cat\).

The tensor unit~\(\Unit\) of \(\Cat\) is also a tensor unit in~\(\Procat\).  Our definition of a locally split extension ensures that~\(\Unit\) is injective.  In general, it is not projective, but this problem disappears if we restrict attention to suitable subcategories such as the subcategory of Fr\'echet spaces, which we identify with a full subcategory of \(\Proban\).

For an algebra~\(A\) in~\(\Cat\) or in~\(\Procat\), let \(\HPchain(A)\) be the \(\Z/2\)\nb-graded chain complex in \(\Procat\) that computes the periodic cyclic homology and cohomology of~\(A\) (see \cites{Cortinas-Valqui:Excision, Cuntz-Quillen:Excision_bivariant}).  Recall that this is always a chain complex of projective systems, even for~\(A\) in~\(\Cat\).

The following theorem is a special case of \cite{Cortinas-Valqui:Excision}*{Theorem 8.1} by Guillermo Corti\~nas and Christian Valqui.

\begin{theorem}
  \label{the:HP_excision_pro}
  Let \(I\mono E\epi Q\) be a locally split extension in~\(\Procat\).  Then the induced maps \(\HPchain(I)\to\HPchain(E)\to\HPchain(Q)\) form a cofibre sequence.  This always yields a cyclic six-term exact sequence for \(\HP^*\).  We also get a cyclic six-term exact sequence for \(\HP_*\) if \(\HPchain(I)\), \(\HPchain(E)\) and \(\HPchain(Q)\) are injective in~\(\Procat\).
\end{theorem}

\begin{corollary}
  \label{cor:HP_excision_Frech_pro}
  Let \(I\mono E\epi Q\) be a pro-locally split extension of Fr\'echet algebras.  Then the induced maps \(\HPchain(I)\to\HPchain(E)\to\HPchain(Q)\) form a cofibre sequence, and there are induced cyclic six-term exact sequences for \(\HP_*\) and \(\HP^*\).
\end{corollary}

\begin{proof}
  The embedding of the category of Fr\'echet spaces into \(\Proban\) is fully faithful and symmetric monoidal, so that it makes no difference in which category we form \(\HPchain\).  Theorem~\ref{the:HP_excision_pro} shows that \(\HPchain(I)\to\HPchain(E)\to\HPchain(Q)\) is a cofibre sequence for the pro-locally split exact category structure and \emph{a fortiori} for the usual exact category structure on~\(\Frech\).  Since~\(\C\) is both injective and projective as an object of~\(\Frech\), this cofibre sequence induces long exact sequences both in homology and in cohomology.
\end{proof}

Corollary~\ref{cor:HP_excision_Frech_pro} applies to all extensions of Fr\'echet algebras with nuclear quotient by Theorem~\ref{the:nuclear_Frechet_locally_split_extension}.

Let~\(Y\) be a closed subset of a smooth manifold~\(X\).  Let \(\Cinf_0(Y;X)\) be the closed ideal of all smooth functions on~\(X\) that vanish on~\(Y\) and let
\[
\Cinf(Y) \defeq \Cinf(X)\bigm/\Cinf_0(Y;X).
\]
We call \(\Cinf(Y)\) the algebra of \emph{smooth functions} on~\(Y\).  By design, the canonical homomorphism from \(\Cinf(Y)\) to the \(\textup C^*\)\nb-algebra of continuous functions on~\(Y\) is injective, that is, \(f=0\) in \(\Cinf(Y)\) if \(f(y)=0\) for all \(y\in Y\).  We get a quotient map \(\Whitney(Y)\to \Cinf(Y)\) because \(\Flat{Y}{X}\) is contained in \(\Cinf_0(Y;X)\).

\begin{theorem}
  \label{the:HP_smooth_singular}
  The quotient map \(\Whitney(Y)\to \Cinf(Y)\) induces a pro-locally split quasi-isomorphism \(\HPchain\bigl(\Whitney(Y)\bigr) \to \HPchain\bigl(\Cinf(Y)\bigr)\) and hence isomorphisms
  \[
  \HP_*\bigl(\Whitney(Y)\bigr) \cong
  \HP_*\bigl(\Cinf(Y)\bigr),\qquad
  \HP^*\bigl(\Whitney(Y)\bigr) \cong
  \HP^*\bigl(\Cinf(Y)\bigr).
  \]
\end{theorem}

Together with Theorem~\ref{the:HH_HC_HP_Whitney}, this yields a formula for the periodic cyclic homology of \(\Cinf(Y)\).

\begin{proof}
  The quotient map \(\Whitney(Y)\to \Cinf(Y)\) is an open surjection with kernel \(N\defeq \Cinf_0(Y;X) \bigm/ \Flat{Y}{X}\).  The extension \(N\mono \Whitney(Y) \epi \Cinf(Y)\) is pro-locally split because \(\Cinf(Y)\) is nuclear.  Corollary~\ref{cor:HP_excision_Frech_pro} shows that \(\HPchain\bigl(\Whitney(Y)\bigr) \to \HPchain\bigl(\Cinf(Y)\bigr)\) is a quasi-isomorphism if and only if \(\HPchain(N)\) is exact.

  We claim that the algebra~\(N\) is topologically nilpotent: if~\(p\) is any continuous semi-norm on~\(N\), then there is \(k\in\N\) such that~\(p\) vanishes on all products of~\(k\) elements in~\(N\).  Since functions in~\(N\) vanish on~\(Y\), products of~\(k\) functions in~\(N\) vanish on~\(Y\) to order~\(k\).  These are annihilated by~\(p\) for sufficiently high~\(k\) because any continuous seminorm on~\(N\) only involves finitely many derivatives.

  Since~\(N\) is topologically nilpotent, the associated pro-algebra \(\diss^*(N)\) is pro-nilpotent in the notation of~\cite{Meyer:HLHA}.  Goodwillie's Theorem \cite{Meyer:HLHA}*{Theorem 4.31} for pro-nilpotent pro-algebras asserts that \(\HPchain(N)\) is contractible.
\end{proof}

Another case where we can prove excision involves ind-algebras.  Let~\(\Cat\) be as above and let~\(\Indcat\) be the category of inductive systems in~\(\Cat\) with the canonical tensor product and the locally split extensions as conflations.  We equip the category of projective systems over~\(\Indcat\) with the induced exact category structure where the conflations are inductive systems of conflations in~\(\Indcat\).

\begin{theorem}
  \label{the:HP_excision_ind}
  Let \(I\mono E\epi Q\) be a locally split extension in~\(\Indcat\).  Then the induced maps \(\HPchain(I)\to\HPchain(E)\to\HPchain(Q)\) form a cofibre sequence.  This yields a cyclic six-term exact sequence for \(\HP_*\).
\end{theorem}

\begin{proof}
  The proof of the Excision Theorem in~\cite{Meyer:Excision} yields this stronger result, as observed in passing in \cite{Meyer:HLHA}*{Remark 4.43}.  In fact, any proof of the excision theorem for split extensions of topological algebras that I know yields this stronger result.  The idea of the argument is as follows.  The proof of the excision theorem for split extensions is potentially constructive in the sense that one can write down an explicit contracting homotopy for the cone of the map \(\HPchain(I)\to \cone\bigl(\HPchain(E)\to\HPchain(Q)\bigr)\).  This explicit formula only uses the multiplication in~\(E\) and the section \(s\colon Q\to E\) of the extension because this is all the data that there is.  The computation checking that this formula works uses only the associativity of the multiplication and the fact that~\(s\) is a section because that is all we know.

  When we have a locally split extension, we may still write down exactly the same formula locally and check that it still works where it is defined.  These locally defined contracting homotopies may be incompatible because we may use different sections on different entries of our inductive system.  Nevertheless, they provide a \emph{local} contracting homotopy and thus establish the desired cofibre sequence.
\end{proof}

Theorem~\ref{the:HP_excision_ind} applies, for instance, to extensions of bornological algebras with nuclear quotient, which become locally split extensions in \(\Indban\).  In particular, this covers the algebra of smooth functions with compact support \(\Cinf_\cpt(Y)\).

\section{Conclusion and outlook}
\label{sec:conclusion}

We formulated and proved a general version of Wodzicki's Excision Theorem for the Hochschild homology of pure algebra extensions with homologically unital kernel in any exact symmetric monoidal category.  The main issue here was to find the right setup to formulate a general Excision Theorem that applies, in particular, to extensions of Fr\'echet algebras with nuclear quotient.

As an application, we computed the continuous Hochschild, cyclic and periodic cyclic homology for the Fr\'echet algebra of Whitney functions on a closed subset of a smooth manifold.

Periodic cyclic homology could satisfy an excision theorem for all pure algebra extensions in \(\Q\)\nb-linear symmetric monoidal categories.  But I do not know how to establish such a general result.  We only considered two special cases involving pro-locally split extensions of pro-algebras and ind-locally split extensions of ind-algebras, where the excision theorem was already known.  This was enough to compute the periodic cyclic homology for algebras of smooth functions on closed subsets of smooth manifolds.

\end{document}